\documentclass{birkjour}
\usepackage{mathrsfs}

\newtheorem{theorem}{Theorem}[section]
\newtheorem{proposition}[theorem]{Proposition}
\newtheorem{lemma}[theorem]{Lemma}
\newtheorem{corollary}[theorem]{Corollary}

\theoremstyle{definition}
\newtheorem{definition}[theorem]{Definition}
\newtheorem{remark}[theorem]{Remark}
\newtheorem{example}[theorem]{Example}

\numberwithin{equation}{section} 

\newcommand{\Rot}{\operatorname{\mathbf{curl}}}
\newcommand{\Div}{\operatorname{\mathrm{div}}}

\newcommand{\rot}{\operatorname{\mathrm{curl}}}

\newcommand{\loc}{\mathrm{loc}}
\renewcommand{\Im}{\operatorname{Im}}

\newcommand  {\C}{{\mathbb C}}
\newcommand  {\N}{{\mathbb N}}

\newcommand  {\R}{{\mathbb R}}

\newcommand {\Id}{\mathrm {I}}

\newcommand  {\D}{\operatorname{ D}}
\newcommand  {\Hess}{\operatorname {Hess}}

\renewcommand{\AA}{\boldsymbol{\mathsf A}}

\newcommand  {\TT}{\boldsymbol{\mathsf T}}

\newcommand  {\HH}{\boldsymbol{\mathsf H}}

\newcommand {\Tr}{\operatorname{Trace}} 
\newcommand {\Jac}{\operatorname{Jac}} 
\newcommand {\cof}{\operatorname{cof}}

\newcommand  {\nn}{\boldsymbol{ n}}

\newcommand  {\uu}{\boldsymbol{ u}}
\newcommand  {\jj}{\boldsymbol{\mathsf j}}

\newcommand  {\vv}{\boldsymbol{ v}}

\newcommand{\transposee}[1]{{\vphantom{#1}}{#1}^{\mathsf T}}  
\usepackage{color}
\newcommand{\Bk}{\color{black}}
\newcommand{\Rd}{\color{red}}
\let\Rd\Bk 

\begin{document}
\hyphenation{ho-mo-ge-ne-ous}

\title[Shape derivatives I: Pseudo-homogeneous kernels]{Shape derivatives  of boundary integral operators in electromagnetic  scattering. Part I:  Shape differentiability of  pseudo-homo\-ge\-ne\-ous boundary integral operators}
\author{Martin Costabel}
  \address{IRMAR, Institut Math\'ematique, Universit\'e de Rennes 1, 35042
    Rennes, France} \email{martin.costabel@univ-rennes1.fr}
\author{Fr\'ed\'erique Le Lou\"er}
 \address{ Institut f\"ur Numerische und Andgewandte Mathematik, Universit\"at G\"ottingen, 37083
   G\"ottingen, Germany} \email{f.lelouer@math.uni-goettingen.de }

\begin{abstract}   In this paper we study the shape differentiability properties of a class of  boundary integral operators and of potentials with weakly singular pseudo-homogeneous kernels acting between classical Sobolev spaces,  with respect to smooth deformations of the boundary. We prove that the boundary integral operators are infinitely differentiable without loss of regularity. The potential operators are infinitely shape differentiable away from the boundary, whereas their derivatives lose regularity near the boundary. We study the shape differentiability of surface differential operators.
The  shape differentiability properties of the usual strongly singular or hypersingular boundary integral operators of interest in  acoustic, elastodynamic or electromagnetic potential theory can then be established by expressing them in terms of  integral operators with weakly singular kernels and of surface differential operators. 
\end{abstract}      

\keywords{Boundary integral operators, pseudo-homogeneous kernels, fundamental solution, surface differential operators, shape derivatives, Sobolev spaces.}

\date{}
\maketitle

\section{Introduction}

Optimal shape design problems and inverse problems involving the scattering of    time-harmonic waves are of practical interest in many important fields of applied physics including  radar and sonar applications,  structural design, bio-medical imaging and non destructive testing. We develop new analytic tools that can be used in algorithms for the numerical solution of such problems.

Shape derivatives are a classical tool in shape optimization  and are  also widely used in  inverse obstacle scattering.  In shape optimization, where extrema of cost functions have to be determined, the analysis of iterative methods  requires the study of the derivative of the solution of  a scattering problem  with respect to the shape of the boundary of the obstacle.  An explicit form of the shape derivatives is required in view of their implementation in iterative algorithms such as gradient methods or Newton's method \cite{DelfourZolesio, PierreHenrot, Zolesio}.  By the method of boundary integral equations,  the shape analysis of the solution of the scattering problem with respect to deformations of the obstacle is obtained from the G\^ateaux differentiability analysis of boundary integral  operators and potentials with weakly singular, strongly singular, or hypersingular kernels. An expression of the shape derivatives of the solution can then be  computed by taking the derivative of its integral representation. This technique  was introduced  for the  Dirichlet and Neumann problems in acoustic scattering by Potthast \cite{Potthast2, Potthast1} and applied to the Dirichlet problem in elastic scattering by Charalambopoulos \cite{Charalambopoulos} in the framework of H\"older continuous and differentiable function spaces.  More recently these results were exploited in acoustic inverse obstacle scattering  to develop novel methods in which a system of nonlinear integral equations  has to be solved by a regularized iterative method \cite{KressRundell,IvanyshynKress2,IvanyshynKress}. 
 
  An extension of the technique to elasticity and electromagnetism requires the shape  differentiability analysis of  the relevant boundary integral operators. More generally, we are concerned  in this paper with the G\^ateaux differentiability of boundary integral operators  with strongly and weakly singular pseudo-homogeneous kernels acting between classical Sobolev spaces,  with respect to smooth deformations of the boundary considered as  a hypersurface of $\R^d$ with $d\in\N$, $d\ge2$. This family of integral operators covers the case of the single and double layer integral operators from the acoustic and the elastic scattering potential theory. The differentiability properties of the hypersingular boundary integral operators  can then be established by expressing them as products of  integral operators with weakly singular kernels and of surface differential operators. In return, however, we  have to study the shape differentiability of surface differential operators. The electromagnetic case presents a specific difficulty: The associated boundary integral operators act as bounded operators on the space of tangential vector fields of mixed regularity  $\TT\HH\sp{-\frac{1}{2}}(\Div_{\Gamma},\Gamma)$. The very definition of the shape derivative of an operator defined on this energy space poses non-trivial problems. This is the subject of the second part of this paper \cite{CostabelLeLouer} where we propose an analysis based on the Helmholtz decomposition \cite{delaBourdonnaye} of   $\TT\HH\sp{-\frac{1}{2}}(\Div_{\Gamma},\Gamma)$.

  This work contains results from the thesis \cite{FLL} where this analysis has been used to construct  and to implement  shape optimization algorithms for dielectric lenses, aimed at obtaining a prescribed radiation pattern. \medskip
 
The paper is organized as follows: 

In Section  \ref{PHkernel} we describe the family of  pseudo-differential boundary integral operators and potentials  that we consider. We use a subclass of the class of pseudo-homogeneous kernels introduced by N\'ed\'elec in his book \cite{Nedelec}. Main results on the regularity of these  operators  are set out. In Section \ref{ShapeD}, we define  the notion of shape derivative and discuss  its connection to G\^ateaux  derivatives. We also recall  elementary results about differentiability in Fr\'echet spaces, following ideas of \cite{DelfourZolesio,DelfourZolesio2} and notations of  \cite{Schwartz}. 
 
Section \ref{GDiffPH} is dedicated to the shape differentiability analysis of the integral operators. 
\Rd
We discuss different definitions of derivatives with respect to deformations of the boundary and compare them to the notions of material derivatives and shape derivatives that are common in continuum mechanics, see Remark~\ref{remmatshape}.
\Bk
We prove that shape derivatives of the boundary integral operators are operators of the same class, that the boundary integral operators are infinitely shape differentiable without loss of regularity, and that the potentials are infinitely shape differentiable away from the boundary of the obstacle, whereas their derivatives lose regularity in the neighborhood of the boundary. A main tool is the proof that the shape differentiability of the integral operators can be reduced to the one of their kernels. We also give higher order G\^ateaux derivatives  of coefficient functions such as the Jacobian of the change of variables associated with the deformation,  or the components of the unit normal vector. These results are new and allow us to obtain explicit forms of higher order derivatives of the integral operators.  A utilization for the implementation of higher order iterative methods is conceivable. 
   
The shape differentiability properties of  usual  surface differential operators is given in the last section. Again we prove their infinite G\^ateaux differentiability and give an explicit expression of their derivatives. These are then  applied to obtain the derivatives of hypersingular boundary integral operators from acoustic, elastic and electromagnetic potential theory. 

Notice that our shape differentiability analysis is realized without restriction to particular classes of deformations of the boundary, such as it is frequently done in the calculus of variations, namely restriction to deformations normal to the surface  as suggested by the structure  theorems for shape derivatives \cite{Hadamard, PierreHenrot,  Zolesio}, or consideration of radial deformations of star-shaped surfaces  \cite{ColtonKress, IvanyshynKress2,IvanyshynKress}.

\section{Pseudo-homogeneous kernels}\label{PHkernel}
Let $\Omega$ denote a bounded domain in $\R^{d}$ with $d\ge2$  and let $\Omega^c$ denote the exterior domain $\R^d\setminus\overline{\Omega}$. 
In this paper, we will assume that the boundary $\Gamma$ of $\Omega$ is a smooth closed hypersurface. Let $\nn$ denote the outer unit normal vector on $\Gamma$.

For a domain $G\subset\R^d$ we denote by $H^s(G)$ the usual $L^2$-based Sobolev space of order $s\in\R$, and by $H^s_{\loc}(\overline G)$ the space of functions whose restrictions to any bounded subdomain $B$ of $G$ belong to $H^s(B)$.

For any $t\in \R$ we denote by $H^t(\Gamma)$ the standard Sobolev space on the boundary $\Gamma$. The dual of $H^t(\Gamma)$  with respect to the $L^2$  scalar product is  $H^{-t}(\Gamma)$. Vector functions and spaces of vector functions will be denoted by boldface letters.

For $\alpha=(\alpha_{1},\hdots,\alpha_{d})\in \N^d$ and $z=(z_{1},\hdots,z_{d})\in\R^d$ we denote  by $\dfrac{\partial^{|\alpha|}}{\partial z^\alpha}$ the linear partial differential operator defined by
$$
\dfrac{\partial^{|\alpha|}}{\partial z^\alpha}=\dfrac{\partial^{\alpha_{1}}}{\partial z_{1}^{\alpha_{1}}}\cdots\dfrac{\partial^{\alpha_{d}}}{\partial z_{d}^{\alpha_{d}}},
$$
where $|\alpha|=\alpha_{1}+\cdots+\alpha_{d}$. For $m\in\N$, 
the total differential of order $m$, a symmetric $m$-linear form on $\R^{d}$,  is denoted by $\D^m$.

\medskip

 The integral operators we consider can be written in the form 
\begin{equation}\label{p}
\mathcal{K}_{\Gamma}u(x)= \int_{\Gamma}k(y,x-y)u(y)d s(y),\quad x\in\Gamma,
\end{equation} 
where the integral is assumed to exist in the sense of  a Cauchy principal value and the kernel  $k$ is regular with respect to the variable  $y\in\Gamma$ and pseudo-homogeneous 
with respect to the variable  $z=x-y\in\R^d$. We recall the regularity properties of these operators on the Sobolev  spaces  $H^t(\Gamma)$ for all $t\in\R$, available also for their adjoint operators
\begin{equation}
\mathcal{K}_{\Gamma}^*(u)(x)=\int_{\Gamma}k(x,y-x)u(y)d s(y),\;x\in\Gamma.
\end{equation}
We use a variant of the class of weakly singular kernels introduced by  N\'ed\'elec  in \cite[pp.~168ff]{Nedelec}. 
 More details can be found in \cite{Eskin, Hormander, Journe, Meyer, Taylor1, Taylor2}.
 
\begin{definition} 
The kernel $G(z)\in\mathscr{C}^{\infty}\big(\R^d\setminus\{0\}\big)$  is said to be \emph{ho\-mo\-ge\-ne\-ous of class} $-m$  for an integer $m\ge0$ if
$$
 \begin{array}{l}\text{(i) for any }\alpha\in \N^d \text{ there is a constant }C_{\alpha}\text{ such that for all }z\in\R^d\setminus\{0\}\\ \text{we have }\qquad\left|\dfrac{\partial^{|\alpha|}}{\partial z^{\alpha}}G(z)\right|\leq C_{\alpha}|z|^{-(d-1)+m-|\alpha|},\\ 
\text{(ii) for any }\alpha\in \N^d\text{ with }|\alpha|=m,\text{ the function }
\dfrac{\partial^{|\alpha|}}{\partial z^{\alpha}}G(z)\text{ is homogeneous}\\\text{ of degree }-(d-1)\text{ with respect to the  variable }z,\\\\
\text{(iii)}\quad \D^mG(z)\text{ is an odd function of $z$}.\end{array}
$$
\end{definition}

\begin{remark}
(i) The number $-m$ in this definition is not the order of homogeneity of the kernel, but related to the order of the corresponding pseudodifferential operator defined on the $d-1$-dimensional manifold $\Gamma$.\\
(ii) Our condition (iii) is stronger than the vanishing condition in Nedelec's original definition, but it is easier to verify, and it is satisfied for the classical integral operators we will be considering. 
\end{remark}

\begin{definition} 
The kernel $k(y,z)$ defined on $\Gamma\times\left(\R^d\setminus\{0\}\right)$   is said to be \emph{pseudo-homo\-gen\-eous of class} $-m$ for an integer $m$ such that  $m\ge0$, if  the kernel $k$ admits the following asymptotic expansion when $z$ tends to $0$:
\begin{equation}
\label{(dev)}
  k(y,z)=\sum_{j\ge0,\ell}b_{m+j}^{\ell}(y)G_{m+j}^{\ell}(z),
\end{equation}
where for $j=0,1,...$ the sum over $\ell$ is finite, $b_{m+j}^{\ell}$ belongs to $\mathscr{C}^{\infty}(\Gamma)$ and $G_{m+j}^{\ell}$ is homogeneous of class $-(m+j)$. 
\end{definition}
In \eqref{(dev)}, one can also consider coefficient functions of the form $b_{m+j}(x,y)$ with $x=y+z$, but using Taylor expansion of such coefficients at $z=0$, we see that this would define the same class of kernels as with \eqref{(dev)}.

\begin{example} (\textbf{Acoustic kernels}) 
\label{exacoustic}
Let $\kappa\in\C\setminus\{0\}$ with $\Im(\kappa)\ge0$ and $d=2$ or $d=3$. The fundamental solution 
$$
 G_{a}(\kappa, z)=
 \left\{\begin{array}{ll}\dfrac{i}{4}H^{(1)}_{0}(\kappa|z|)&\text{when }d=2\\
 \dfrac{e^{i\kappa|z|}}{4\pi|z|}&\text{when }d=3 
 \end{array}\right.
$$  
of the Helmholtz equation $\Delta u+\kappa^2 u=0$ in $\R^d$   is pseudo-homogeneous of class $-1$.  Its normal derivative $\frac{\partial}{\partial \nn(y)}G_{a}(\kappa, z)$ is {\it a priori} pseudo-homogeneous of class 0 but one can show  that in the case of smooth boundaries it is a pseudo-homogeneous kernel  of class -1. 

Indeed one can write $$\dfrac{e^{i\kappa|z|}}{4\pi|z|}=\frac{1}{|z|}+i\kappa-\frac{\kappa^2}{2}|z|-\frac{i\kappa^3}{6}|z|^2+\hdots$$The first term is homogeneous of class $-1$, the second term is smooth and for $j\ge3$ the $j$-th term is homogeneous of class $-(1+j)$. The double layer kernel has the expansion 
$$
 \frac{\partial}{\partial \nn(y)}G_{a}(\kappa, z)=\nn(y)\cdot\nabla^zG_{a}(\kappa, z)=(\nn(y)\cdot z)\left(\!-\frac{1}{|z|^3}-\frac{\kappa^2}{2}\frac{1}{|z|}-\frac{i\kappa^3}{3}+\hdots\right).
$$ 
One can prove  that the function $g(x,y)={\nn(y)\cdot(x-y)}$ behaves as $|x-y|^2$ when $z=x-y\to0$ (see for instance \cite[p.\ 173]{Nedelec}). We refer to example \ref{Dk} for a proof using a local coordinate system.
\end{example}

\begin{example} (\textbf{Elastodynamic kernels}) 
\label{exelastic}
Let $\omega\in\R$ and $d=2$ or $d=3$. Denote by $\rho,\mu$ and $\lambda$ the density and Lam\'e's constants.  The symmetric fundamental solution of the Navier equation $-\mu\Delta\uu-(\mu+\lambda)\nabla\Div\uu-\rho w^2 \uu=0$, given by 
$$
  G_{e}(\kappa_{s},\kappa_{p}, z)=\frac{1}{\mu}\left(G_{a}(\kappa_{s},z)\cdot\Id_{\R^d}+\frac{1}{\kappa_{s}^2}\Hess\big(G_{a}(\kappa_{s},z)-G_{a}(\kappa_{p},z)\big)\right),
$$ with $\kappa_{s}=\omega\sqrt{\frac{\rho}{\mu}}$ and $\kappa_{p}=\omega\sqrt{\frac{\rho}{\lambda+2\mu}}$,   is pseudo-homogeneous of class $-1$ . The traction operator is defined by 
$$T\uu=2\mu\frac{\partial\uu}{\partial\nn}+\lambda\big(\Div\uu\big)\nn+\mu\,\nn\wedge\Rot\uu.$$
The double layer kernel  $ \transposee{\big(T_{y}G_{e}(\kappa_{s},\kappa_{p}, x-y)\big)}$ is pseudo-homogeneous of class 0. The index $y$ of $T_{y}$ means that the differentiation is with respect to the variable $y$. Notice that $T_{y}G_{e}(\kappa_{s},\kappa_{p}, x-y)$ is the tensor obtained by applying the traction operator $T_{y}$ to each column of $G_{e}(\kappa_{s},\kappa_{p}, x-y)$. 
\end{example}  

 For the proof of the following theorem we refer to \cite{Nedelec, Taylor2}. 
\begin{theorem}Let $k$ be a  pseudo-homogeneous kernel of class $-m$%
. The  associated boundary integral operator $\mathcal{K}_{\Gamma}$ given by \eqref{p} is linear and continuous from $H^{t}(\Gamma)$ to $H^{t+m}(\Gamma)$ for all $t\in\R$.
The same result is true for the adjoint operator $\mathcal{K}_{\Gamma}^{*}$.\newline
\end{theorem}

 The following theorem is established in \cite{Eskin}.
\begin{theorem}
 Let $s\in\R$. Let $k$ be a  pseudo-homogeneous kernel of class $-m$.\\
  The potential operator  $\mathcal{P}$ defined by 
\begin{equation}
\label{p'}
\mathcal{P}(u)(x)=\int_{\Gamma}k(y,x-y)u(y)d s(y),\quad x\in\R^d\setminus\Gamma\end{equation}
is linear and continuous from $H^{s}(\Gamma)$ to 
$H^{s+m+\frac{1}{2}}(\Omega)\cup H^{s+m+\frac{1}{2}}_{loc}(\overline{\Omega^c})$.
\end{theorem}

\section{Some remarks on shape derivatives} \label{ShapeD}

We want to study the dependence of operators defined by integrals over the boundary $\Gamma$  on the geometry of $\Gamma$. This dependence is highly nonlinear. The usual  tools of differential calculus require the framework of  topological vector  spaces which are locally convex at least, a framework that is not immediately present in the case of shape functionals. The standard approach consists in representing the variations of the domain $\Omega$ by elements of a  function space. We consider  variations generated by transformations of the form 
$$
  x\mapsto x+r(x)
  $$ 
of point $x$ in the space $\R^d$, where $r$ is a smooth vector function defined in the neighborhood of $\Gamma$. This transformation deforms the domain $\Omega$  into a  domain $\Omega_{r}$ with boundary $\Gamma_{r}$. The functions $r$ are assumed to be  sufficiently small  elements of the Fr\'echet space $\mathcal{X}=\mathscr{C}^{\infty}(\Gamma,\R^d)$ in order that $(\Id+r)$ is a diffeomorphism from $\Gamma$ to  $$\Gamma_{r}=(\Id+r)\Gamma=\left\{x_{r}=x+r(x); x\in\Gamma\right\}.$$
 For $\varepsilon$ small enough we set 
$$
  B^{\infty}(0,\varepsilon)= \left\{r\in\mathscr{C}^{\infty}(\Gamma,\R^d),\; d_{\infty}(0,r)<\varepsilon\right\},
$$
where $d_{\infty}$ is the distance induced by the family of non-decreasing norms $(\|\cdot\|_{k})_{k\in\N}$ defined by
$$
 \|r\|_{k}=\sup_{0\leq m\leq k}\;\sup_{x\in\R^d}\left|\D^mr(x)\right|.
$$
 
 Consider a mapping $F$ defined on the set $\{\Gamma_{r};\; r\in B^{\infty}(0,\varepsilon)\}$ of boundaries. We introduce a new mapping 
$$
  B^{\infty}(0,\varepsilon)\ni r\mapsto\mathcal{F}_{\Gamma}(r)=F(\Gamma_{r}).
$$ 
We define the shape derivative of the mapping $F$ through the transformation $\Gamma\ni x\mapsto x+\xi(x)\in\R^d$ by  \begin{equation}\label{defSh}d F[\Gamma;\xi]:=\lim_{t\to0}\dfrac{F(\Gamma_{t\xi})-F(\Gamma)}{t}=\lim_{t\to0}\dfrac{\mathcal{F}_{\Gamma}(t\xi)-\mathcal{F}_{\Gamma}(0)}{t}\end{equation} if the limit  exists and is finite. The shape derivatives of $F$ are related to the G\^ateaux derivatives of $\mathcal{F}_{\Gamma}$ (see \cite{PierreHenrot,Zolesio}).

Fix $r_{0}\in B^{\infty}(0,\varepsilon)$. Following the same procedure, one can construct another mapping $\mathcal{F}_{\Gamma_{r_{0}}}$ defined  on the family of boundaries $$\{(\Id+r')(\Gamma_{r_{0}});\; r'\in B^{\infty}(0,\varepsilon')\}.$$
 Notice that $\mathcal{F}_{\Gamma_{r_{0}}}(0)=F(\Gamma_{r_{0}})=\mathcal{F}_{\Gamma}(r_{0})$ and  $\mathcal{F}_{\Gamma_{r_{0}}}((r-r_{0})\circ(\Id+r_{0})^{-1})=F(\Gamma_{r})=\mathcal{F}_{\Gamma}(r)$.

\subsection{Differentiability in Fr\'echet spaces: elementary results}

Fr\'echet spaces are locally convex, metrisable and complete  topological vector spaces  on which the differential calculus available on Banach spaces can be extended. We recall some of the results. We refer to Schwartz's book \cite{Schwartz} for more details.
\medskip

Let $\mathcal{X}$ and $\mathcal{Y}$ be Fr\'echet spaces and let $U$ be a subset of $\mathcal{X}$.
\begin{definition}(\textbf{G\^ateaux semi-derivatives}) 
 The mapping  $f:U\rightarrow \mathcal{Y}$ is said to have  a G\^ateaux semiderivative at $r_{0}\in U$ in the direction of $\xi\in\mathcal{X}$ if the following limit exists in $\mathcal{Y}$
$$
  df[r_{0};\xi]=\lim_{t\rightarrow0}\dfrac{f(r_{0}+t\xi)-f(r_{0})}{t}=\frac{d}{dt}_{\big|t=0}f(r_{0}+t\xi).
$$
\end{definition}

\begin{definition}(\textbf{G\^ateaux differentiability}) 
 The mapping  $f:U\rightarrow \mathcal{Y}$ is said to be  G\^ateaux differentiable at $r_{0}\in U$ if it has G\^ateaux semiderivatives in all directions $\xi\in\mathcal{X}$ and if the mapping 
$$
  \mathcal{X}\ni \xi\mapsto df[r_{0};\xi]\in\mathcal{Y}
$$ 
is linear and continuous.
\end{definition}
 We say that $f$ is continuously (or $\mathscr{C}^1$-) G\^ateaux differentiable if it is  G\^ateaux differentiable at all $r_{0}\in U$ and the mapping 
$$
   U\times\mathcal{X}\ni df:(r_{0},\xi)\mapsto df[r_{0};\xi]\in\mathcal{Y}
$$
is continuous.

\begin{remark}In the calculus of shape derivatives, we usually consider  the G\^ateaux derivative at $r=0$  only. This is due to the result: If $\mathcal{F}_{\Gamma}$ is  G\^ateaux differentiable on $B^{\infty}(0,\varepsilon)$, then for all $\xi\in\mathcal{X}$ we have
$$
  d\mathcal{F}_{\Gamma}[r_{0};\xi]=dF[\Gamma_{r_{0}};\xi\circ(\Id+r_{0})^{-1}]=d\mathcal{F}_{\Gamma_{r_{0}}}[0;\xi\circ(\Id+r_{0})^{-1}].
$$
\end{remark}

\begin{definition}(\textbf{Higher order derivatives}) 
Let $m\in\N$.  We say that $f$ is $(m+1)$-times continuously (or $\mathscr{C}^{m+1}$-) G\^ateaux differentiable  if it is $\mathscr{C}^{m}$-G\^ateaux differentiable  and 
$$
 U\ni r\mapsto d^mf[r;\xi_{1},\hdots,\xi_{m}]
$$ 
is continuously G\^ateaux differentiable for all $m$-tuples $(\xi_{1},\hdots,\xi_{m})\in\mathcal{X}^{m}$. Then for all $r_{0}\in U$ the  mapping
$$
  \mathcal{X}^{m+1}\ni(\xi_{1},\hdots,\xi_{m+1})\mapsto d^{m+1}f[r_{0};\xi_{1},\hdots,\xi_{m+1}]\in\mathcal{Y}
$$ 
is $(m+1)$-linear, symmetric and continuous.   We say that $f$ is $\mathscr{C}^\infty$-G\^ateaux differentiable if it is $\mathscr{C}^m$-G\^ateaux differentiable for all $m\in\N$.
\end{definition}
\begin{proposition}
 \label{curve} 
Let $f:U\rightarrow\mathcal{Y}$ be $\mathscr{C}^{m}$-G\^ateaux differentiable. Let us fix  $r_{0}\in U$ and $\xi\in\mathcal{X}$. We set $\gamma(t)=f(r_{0}+t\xi)$.

i)  The function of a real variable $\gamma$ is of class $\mathscr{C}^m$ in the neighborhood of zero and
\begin{equation}
\label{GCm}
\gamma^{(m)}(t)=\frac{d^m}{d t^m}_{\big|t=0}f(r_{0}+t\xi)=d^mf[r_{0};\underbrace{\xi,\hdots,\xi}_{m\text{ times}}].
\end{equation}

 ii) We use the notation 
$$\frac{\partial^m}{\partial r^m}f[r_{0};\xi]=d^mf[r_{0};\underbrace{\xi,\hdots,\xi}_{m\text{ times}}].$$
 We then have
\begin{equation}
\label{Gsym}
d^mf[r_{0};\xi_{1},\hdots,\xi_{m}]=\frac{1}{m!}\sum_{p=1}^m (-1)^{m-p}\sum_{1\leq i_{1}<\cdots<i_{p}\leq m}\dfrac{\partial^m}{\partial r^m}f[r_{0};\xi_{i_{1}}+\hdots+\xi_{i_{p}}].
\end{equation}
\end{proposition}
 Thus the knowledge of $\dfrac{\partial^m}{\partial r^m}f[r_{0};\xi]$  suffices to determine the expression of $d^mf[r_{0};\xi_{1},\hdots,\xi_{m}]$.
 
 \begin{proposition} 
 Let $f:U\rightarrow\mathcal{Y}$ be $\mathscr{C}^{m}$-G\^ateaux differentiable. Let us fix  $r_{0}\in U$ and $\xi\in\mathcal{X}$ with $\xi$ sufficiently small. Then we have the following Taylor expansion with integral remainder : 
 $$
 f(r_{0}+\xi)=\sum_{k=1}^{m-1}\frac{1}{k!}\frac{\partial^k}{\partial r^k}f[r_{0};\xi]+\int_{0}^1\frac{(1-\lambda)^m}{m!}\frac{\partial^{m}}{\partial r^{m}}f[r_{0}+\lambda \xi;\xi]d\lambda.
 $$
 \end{proposition}
  The chain and product rules  are still available for  $\mathscr{C}^m$-G\^ateaux differentiable maps between Fr\'echet spaces.

\section{Shape differentiability of boundary integral operators} \label{GDiffPH}
 Let $x_{r}$ denote an element of $\Gamma_{r}$ and let $\nn_{r}$ be the outer unit  normal vector to  $\Gamma_{r}$. When $r=0$ we write $\nn_{0}=\nn$. We denote by $d s(x_{r})$ the area element on $\Gamma_{r}$.

 In this  section we want to establish the  differentiability properties with respect to  $r\in B^{\infty}(0,\varepsilon)$ of   boundary integral operators   $\mathcal{K}_{\Gamma_{r}}$ defined for a function $u_{r}\in H^{t}(\Gamma_{r})$ by:
\begin{equation}\left(\mathcal{K}_{\Gamma_{r}}u_{r}\right)(x_{r})=\int_{\Gamma_{r}}k_{r}(y_{r},x_{r}-y_{r})u_{r}(y_{r})d s(y_{r}),\;x_{r}\in\Gamma_{r}\end{equation}
and of  potential operators   $\mathcal{P}_{r}$  defined  by:
\begin{equation}\left(\mathcal{P}_{r}u_{r}\right)(x)=\int_{\Gamma_{r}}k_{r}(y_{r},x-y_{r})u_{r}(y_{r})d s(y_{r}),\;x\in \Omega_{r}\cup\Omega_{r}^c,\end{equation}
  where $k_{r}\in\mathscr{C}^{\infty}\left(\Gamma_{r}\times\left(\R^d\setminus\{0\}\right)\right)$ is a pseudo-homogeneous kernel of class $-m$ with $m\in\N$.  
  \medskip
  
We point out that we have to analyze mappings of the form $r\mapsto \mathcal{F}_{\Gamma}(r)$ where the  domain of definition of $\mathcal{F}_{\Gamma}(r)$ varies with $r$. This is the main difficulty encountered in the calculus of shape variations. We propose  different strategies according to the definition  of the mapping $\mathcal{F}_{\Gamma}$.

\noindent
(i) A first idea, quite classical (see \cite{PierreHenrot, Potthast2, Potthast1}), is that  instead of  studying mappings $r \mapsto \mathcal{F}_{\Gamma}(r)$ where $\mathcal{F}_{\Gamma}(r)=u_{r}$ is a \emph{function} defined on the boundary $\Gamma_{r}$, we consider the mapping 
$$
   r \mapsto u_{r}\circ(\Id+r).
$$ 
Typical examples of such functions $u_{r}$ are  the normal vector $\nn_{r}$ on $\Gamma_{r}$  and the kernel $k_{r}$ of a boundary integral operator $\mathcal{K}_{\Gamma_{r}}$ (see Examples \ref{exacoustic} and \ref{exelastic}).
\smallskip

To formalize this, we define the transformation (``pullback'') $\tau_{r}$  which maps a function $u_{r}$ defined on $\Gamma_{r}$ to the function $u_{r}\circ(\Id+r)$ defined on $\Gamma$. For all $r\in B^{\infty}(0,\varepsilon)$, the transformation $\tau_{r}$ is linear and continuous from the function spaces $\mathscr{C}^k(\Gamma_{r})$ and  $H^t(\Gamma_{r})$ to $\mathscr{C}^k(\Gamma)$ and $H^t(\Gamma)$, respectively, and admits an inverse. We have
$$
 (\tau_{r}u_{r})(x)=u_{r}(x+r(x))\text{ and }(\tau_{r}^{-1}u)(x_{r})=u(x).
$$

\noindent
(ii) Next, for linear bounded \emph{operators} between function spaces on the boundary, we use conjugation with the pullback $\tau_{r}$: 
Instead of studying the mapping
$$
 B^{\infty}(0,\varepsilon)\ni r\mapsto \mathcal{F}_{\Gamma}(r)=\mathcal{K}_{\Gamma_{r}}\in\mathscr{L}\left(H^s(\Gamma_{r}),H^{s+m}(\Gamma_{r})\right)
$$ 
we  consider the  mapping
$$
 B^{\infty}(0,\varepsilon)\ni r\mapsto \tau_{r}\mathcal{K}_{\Gamma_{r}}\tau_{r}^{-1}\in\mathscr{L}\left(H^s(\Gamma),H^{s+m}(\Gamma)\right).
 $$
We have for $u\in H^s(\Gamma)$ and $x\in\Gamma$:
\begin{equation}
\label{boundr}
\big(\tau_{r}\mathcal{K}_{\Gamma_{r}}\tau_{r}^{-1}\big)(u)(x)
=\int_{\Gamma}k_{r}\big(y+r(y),x+r(x)-y-r(y)\big)\,u(y)\,J_{r}(y)\, d s(y),
\end{equation} 
where $J_{r}$ is the Jacobian (the determinant of the Jacobian matrix) of the change of variables on the surface, mapping $x\in\Gamma$ to $x+r(x)\in\Gamma_{r}$. 
\medskip

\noindent
(iii) The third case concerns \emph{potential operators} acting from the boundary to the domain:\\
Each domain $\Omega$ is a countable union of compact subsets: $\Omega=\bigcup\limits_{p\in\N}K_{p}$. For all $p\in\N$, there exists  $\varepsilon_{p}>0$ such that 
$K_{p}\subset \bigcap\limits_{r\in B(0,\varepsilon_{p})}\Omega_{r}$. 
Thus, instead of studying the mapping
$$
 B^{\infty}(0,\varepsilon)\ni r\mapsto \mathcal{F}_{\Gamma}(r)=\mathcal{P}_{r}\in\mathscr{L}\left(H^{s}(\Gamma_{r}),H^{s+m+\frac{1}{2}}(\Omega_{r})\right)
$$ 
we can consider the mapping
$$
 B^{\infty}(0,\varepsilon_{p})\ni r\mapsto \mathcal{P}_{r}\tau_{r}^{-1}\in\mathscr{L}\left(H^{s}(\Gamma),H^{s+m+\frac{1}{2}}(K_{p})\right).
$$
 We have for $u\in H^{s}(\Gamma)$
\begin{equation}
\label{potr}
\big(\mathcal{P}_{\Gamma}(r)\tau_{r}^{-1}\big)(u)(x)
 =\int_{\Gamma}k_{r}\big(y+r(y),x-y-r(y)\big)\,u(y)\,J_{r}(y)\, d s(y),\quad x\in K_{p}.
\end{equation}
 Then passing to the limit $p\to\infty$ we can deduce the differentiability properties of the potentials on the whole domain $\Omega$. We use the analogous technique for the exterior domain $\Omega^c$.
 
In the framework of boundary integral equations, these approaches were introduced by Potthast  \cite{Potthast2, Potthast1} in order to study the shape differentiability of solutions of acoustic boundary value problems.

\Rd
\begin{remark}
\label{remmatshape}
In continuum mechanics, when the deformation $x\mapsto r(x)=r_{0}(x)+t\xi(x)$ is interpreted as a flow with initial velocity field $\xi(x)$, one frequently considers two different derivatives of functions $u_{r}$ defined on $\Omega_{r}$. 
The \emph{material derivative} $\dot{u}_{r}$ is computed by pulling $u_{r}$ back to the reference domain $\Omega$, thus by differentiating 
$r \mapsto \tau_{r}u_{r}=u_{r}\circ(\Id+r)$.
The \emph{shape derivative} $u'_{r}(x)$ at a point $x$ is defined by differentiating $u_{r}(x)$ directly. At $r=0$ the difference between the two derivatives is a convection term:
\begin{equation}
\label{matshape}
  \dot{u}_{0} = u'_{0} + \xi\cdot\nabla u_{0}\,.
\end{equation}
This is easily seen from the definition of the material derivative
$$
 \dot{u}_{r}(x)=d(\tau u)[0;\xi](x) = \frac{d}{dt}_{\big|t=0}u_{t\xi}(x+t\xi(x))
 = d\,u[0;\xi](x) + \xi(x)\cdot\nabla u_{0}(x)\,.
$$
Relation \eqref{matshape} can be used to compute the shape derivative from the simpler material derivative, see \cite{LeugeringetalAMOptim11} for an application.

In this terminology, the derivatives of boundary functions and operators in (i) and (ii) above would be analogous to material derivatives, whereas the derivatives of potentials in (iii) correspond to shape derivatives. Instead of formally defining the terms ``material derivative'' and ``shape derivative'', we prefer here to explain in each instance precisely which G\^ateaux derivative is meant. We want to emphasize, however, that the \emph{shape derivatives} of solutions of electromagnetic transmission problems can be obtained by using the three kinds of derivatives defined above. This will be explained in detail in Part II of this work. The construction is based on an integral representation of the solution of the transmission problem by potentials, the densities of which are solutions of boundary integral equations with operators of the type studied here. Thus the mapping from the given right hand side to the solution is a composition of boundary integral operators, inverses of boundary integral operators, and potential operators. By the chain rule, its derivative is then obtained by composing boundary integral operators, their inverses, and potential operators with derivatives of type (i), (ii), and (iii) above. The same structure gives the shape gradient of shape functionals that are defined from the solution of the transmission problem. In this case, also adjoints of the boundary integral operators have to be differentiated. This poses no new problem, because adjoints of operators with quasi-homogeneous kernels have quasi-homogeneous kernels, too.
\end{remark}
\Bk

\subsection{G\^ateaux differentiability of coefficient functions}

For the  analysis of the integral operators defined by \eqref{boundr} and \eqref{potr}, we first have to analyze coefficient functions such as the Jacobian of the change of variables $\Gamma\ni x\mapsto x+r(x)\in\Gamma_{r}$, or the normal vector $\nn_{r}$ on $\Gamma_{r}$.

We use the standard surface differential operators as described in detail in \cite{Nedelec}.
For a vector function $\vv\in\mathscr{C}^k(\R^d,\C^d)$ with $k\in\N^*$, we denote by $[\nabla\vv]$ the matrix the $i$-th column of which is the gradient of the $i$-th component of $\vv$, and we write $[\D\vv]=\transposee{[\nabla\vv]}$.
The tangential gradient of a scalar function $u\in\mathscr{C}^k(\Gamma,\C)$ is defined by 
\begin{equation}\label{G}\nabla_{\Gamma}u=\nabla\tilde{u}_{|\Gamma}-\left(\nabla\tilde{u}_{|\Gamma}\cdot\nn\right)\nn,\end{equation}
where $\tilde{u}$ is an extension of $u$ to the whole space $\R^d$.
For a vector function $\uu\in\mathscr{C}^k(\Gamma,\C^d)$, we again denote by $[\nabla_{\Gamma}\uu]$ the matrix the $i$-th column of which is the tangential gradient of the $i$-th component of $\uu$ and we set  $[\D_{\Gamma}\uu]=\transposee{[\nabla_{\Gamma}\uu]}$.

We define the surface divergence of a vector function $\uu\in\mathscr{C}^k(\Gamma,\C^d)$   by 
\begin{equation}
\Div_{\Gamma}\uu=\Div\tilde{\uu}_{|\Gamma}-\left([\nabla\tilde{\uu}_{|\Gamma}]\nn\cdot\nn\right)=\Div\tilde{\uu}_{|\Gamma}-\left(\nn\cdot\frac{\partial \uu}{\partial\nn}\right),
\end{equation}
where $\tilde{\uu}$ is an extension of $\uu$ to the whole space $\R^d$.
These definitions do not depend on the choice of the extension. 

The surface Jacobian $J_{r}$ is given by the formula
$
 J_{r}=\Jac_{\Gamma}(\Id+r)=\|w_{r}\| 
$ with 
$$
 w_{r}
 =\cof(\Id+\D r_{|\Gamma})\nn=\det(\Id+\D r_{|\Gamma})\transposee{(\Id+\D r_{|\Gamma})^{-1}}\nn,
$$  
 where $\cof(A)$ means the matrix of cofactors of the matrix $A$,
 and  the normal vector $\nn_{r}$ is given by
$$
 \nn_{r}=\tau_{r}^{-1}\left(\frac{w_{r}}{\|w_{r}\|}\right).
$$
The first derivative at $r=0$ of these functions are well known, we refer for instance  to Henrot--Pierre  \cite{ PierreHenrot}. Here we present a method that allows to obtain higher order derivatives.

\begin{lemma}
\label{J} 
The functional $\mathcal{J}$ mapping $r\in B^{\infty}(0,\varepsilon)$ to the Jacobian $J_{r}\in\mathscr{C}^{\infty}(\Gamma,\R)$  is $\mathscr{C}^{\infty}$-G\^ateaux differentiable and its first derivative  at $r_{0}$ is given for $\xi\in\mathscr{C}^{\infty}(\Gamma,\R^d)$ by
$$
 d\mathcal{J}[r_{0},\xi]=J_{r_{0}}\big(\tau_{r_{0}}\Div_{\Gamma_{r_{0}}}(\tau^{-1}_{r_{0}}\xi)\big).
$$ 
\end{lemma}
\begin{proof} 
We just have to prove the $\mathscr{C}^{\infty}$-G\^ateaux differentiability of 
$$
 \mathcal{W}:B^{\infty}(0,\varepsilon)\ni r\mapsto w_{r}=\cof(\Id+\D r_{|\Gamma})\nn\in\mathscr{C}^{\infty}(\Gamma).
$$  
We use a local coordinate system. Assume that $\Gamma$ is parametrized by an atlas  $(\mathcal{O}_{i},\phi_{i})_{1\leq i\leq p}$  then $\Gamma_{r}$ can be parametrized by the atlas $(\mathcal{O}_{i},(\Id+r)\circ\phi_{i})_{1\leq i\leq p}$. For any $x\in\Gamma$, let us denote by $e_{1}(x),e_{2}(x),\hdots,e_{d-1}(x)$ a vector basis of the tangent plane to $\Gamma$ at $x$. A basis of the tangent plane to $\Gamma_{r}$ at $x+r(x)$ is then given by 
$$
 e_{i}(r,x)=[(\Id+\D r)(x)]e_{i}(x)\quad\text{for }i=1,\hdots,d-1.
$$ 
Notice that for $i=1,\hdots,d-1$ the mapping $B^{\infty}(0,\varepsilon)\ni r\mapsto e_{i}(r)\in\mathscr{C}^{\infty}(\Gamma,\R^d)$  is $\mathscr{C}^{\infty}$-G\^ateaux differentiable. Its first derivative is
$de_{i}[r_{0};\xi]=[\D\xi]e_{i}(r_{0})$, and higher order derivatives vanish.
We have 
$$
  w_{r}(x)=\dfrac{\bigwedge\limits_{i=1}^{d-1} e_{i}(r,x)}{\left|\bigwedge\limits_{i=1}^{d-1}e_{i}(x)\right|},
$$
where the wedge means the exterior product.
Since the mappings $r\mapsto e_{i}(r)$, for $i=1,\hdots,d-1$ are $\mathscr{C}^{\infty}$-G\^ateaux differentiable, by composition the mapping $W$ is, too.  We compute now the derivatives using formulas \eqref{GCm}-\eqref{Gsym}. Let $\xi\in\mathscr{C}^{\infty}(\Gamma,\R^d)$ and $t$ small enough. We have at $r_{0}\in B^{\infty}(0,\varepsilon)$

$$
\frac{\partial^m \mathcal{W}}{\partial r^m}[r_{0},\xi]=\frac{\partial^m}{\partial t^m}_{\Big| t=0}\,\frac{\bigwedge\limits_{i=1}^{d-1}(\Id+Dr_{0}+tD\xi)e_{i}(x)}{\left|\bigwedge\limits_{i=1}^{d-1}e_{i}(x)\right|}.
$$
To simplify this expression one notes that 
$$
 \begin{aligned}{}
 [\D\xi(x)]e_{i}(x)&=[\D\xi(x)][(\Id+\D r_{0})(x)]^{-1}[(\Id+\D r_{0})(x)]e_{i}(x)\\
 &=[\D\xi(x)][\D(\Id+ r_{0})^{-1}(x+r_{0}(x))][(\Id+\D r_{0})(x)]e_{i}(x)\\
 &=[\tau_{r_{0}}\D(\tau_{r_{0}}^{-1}\xi)(x)]e_{i}(r_{0},x)=[\tau_{r_{0}}\D_{\Gamma_{r_{0}}}(\tau_{r_{0}}^{-1}\xi)(x)]e_{i}(r_{0},x).
\end{aligned}
$$
Now given a $(d\times d)$ matrix $A$  we have 
$$
 \sum_{i=1}^{d-1}\cdots \wedge e_{i-1}\times A e_{i}\wedge e_{i+1}\wedge\cdots=(\Tr(A)\Id-\transposee{A})\bigwedge_{i=1}^{d-1}e_{i}.
$$
 Thus we have with $A=[\tau_{r_{0}}\D_{\Gamma_{r_{0}}}(\tau_{r_{0}}^{-1}\xi)]$ and $B_{0}=\Id$, $B_{1}(A)=\Tr(A)\Id-\transposee{A}$
$$
(\#)
\left\{\begin{array}{ccl}
 \mathcal{W}(r_{0})&=&J_{r_{0}}(\tau_{r_{0}}\nn_{r_{0}}),\\
 \dfrac{\partial \mathcal{W}}{\partial r}[r_{0},\xi]&=&J_{r_{0}}\Big(\big(\tau_{r_{0}}\Div_{\Gamma_{r_{0}}}(\tau_{r_{0}}^{-1}\xi)\big)\tau_{r_{0}}\nn_{r_{0}}\\
 &&\qquad -\big[\tau_{r_{0}}\nabla_{\Gamma_{r_{0}}}(\tau_{r_{0}}^{-1}\xi)\big]\tau_{r_{0}}\nn_{r_{0}}\Big)\\&=&[B_{1}(A)\xi]\mathcal{W}(r_{0}),\\
\dfrac{\partial^{m} \mathcal{W}}{\partial r^{m}}[r_{0},\xi]&=&[B_{m}(A)\xi]\mathcal{W}(r_{0})\\
&=&\sum\limits_{i=1}^{m}(-1)^{i+1}\dfrac{(m-1)!}{(m-i)!}[B_{1}(A^{i})B_{m-i}(A)\xi]\mathcal{W}(r_{0})\\
&&\text{ for }1\leq m\leq d-1\\
\dfrac{\partial^{m} \mathcal{W}}{\partial r^{m}}[r_{0},\xi]&\equiv&0\text{ for all }m\ge d.
\end{array}\right.
$$

It follows that 
$$
\begin{aligned}
\frac{\partial \mathcal{J}}{\partial r}[r_{0},\xi]&=\frac{1}{\|\mathcal{W}(r_{0})\|}\frac{\partial \mathcal{W}}{\partial r}[r_{0},\xi]\cdot \mathcal{W}(r_{0}) \\
&=\frac{\partial \mathcal{W}}{\partial r}[r_{0},\xi]\cdot\tau_{r_{0}}\nn_{r_{0}}=J_{r_{0}}\big(\tau_{r_{0}}\Div_{\Gamma_{r_{0}}}(\tau^{-1}_{r_{0}}\xi)\big).
\end{aligned}
$$

\end{proof}
From $(\#)$ we deduce  easily the G\^ateaux differentiability of $r\mapsto \tau_{r}\nn_{r}$. 
\begin{lemma} 
 \label{N} 
 The mapping $\mathcal{N}$ from $r\in B^{\infty}(0,\varepsilon)$ to $\tau_{r}\nn_{r}=\nn_{r}\circ(\Id+r)\in\mathscr{C}^{\infty}(\Gamma,\R^d)$ is $\mathscr{C}^{\infty}$-G\^ateaux-differentiable and its first derivative  at $r_{0}$ is  defined for $\xi\in\mathscr{C}^{\infty}(\Gamma,\R^d)$ by:  
$$
 \frac{\partial \mathcal{N}}{\partial r}[r_{0},\xi]=-\left[\tau_{r_{0}}\nabla_{\Gamma_{r_{0}}}(\tau_{r_{0}}^{-1}\xi)\right]\mathcal{N}(r_{0}).
$$ 
\end{lemma}
\begin{proof} Using the preceding proof, we find
$$
\begin{array}{ccl}
\!\!\!\dfrac{\partial \mathcal{N}}{\partial r}[r_{0},\xi]&=& \dfrac{1}{\|\mathcal{W}(r_{0})\|}\dfrac{\partial \mathcal{W}}{\partial r}[r_{0},\xi]-\dfrac{1}{\|\mathcal{W}(r_{0})\|^3}\left(\dfrac{\partial \mathcal{W}}{\partial r}[r_{0},\xi]\cdot \mathcal{W}(r_{0})\right) \mathcal{W}(r_{0})\vspace{2mm}\\
&=&J_{r_{0}}^{-1}\left(\dfrac{\partial \mathcal{W}}{\partial r}[r_{0},\xi]-\left(\dfrac{\partial \mathcal{W}}{\partial r}[r_{0},\xi]\cdot(\tau_{r_{0}}\nn_{r_{0}})\right)\right)\tau_{r_{0}}\nn_{r_{0}}\vspace{2mm}\\
&=&-\left[\tau_{r_{0}}\nabla_{\Gamma_{r_{0}}}(\tau_{r_{0}}^{-1}\xi)\right]\tau_{r_{0}}\nn_{r_{0}}.
\end{array}
$$
\end{proof}
To obtain higher order shape derivatives of these mappings one can use the equalities $(\#)$ and 
$$
(*)\left\{\begin{array}{ccl}\|\tau_{r}\nn_{r}\|&\equiv&1,\vspace{2mm}\\
\dfrac{\partial^m \mathcal{N}\cdot \mathcal{N}}{\partial r^m}[r_{0},\xi]&\equiv&0\text{ for all }m\ge1.\end{array}\right.
$$
For example, we have at $r=0$ in the direction $\xi\in\mathscr{C}^{\infty}(\Gamma,\R^d)$:
$$
\frac{\partial \mathcal{J}}{\partial r}[0,\xi]=\Div_{\Gamma}\xi \text{ and }\frac{\partial \mathcal{N}}{\partial r}[0,\xi]=-[\nabla_{\Gamma}\xi]\nn.
$$
Using Proposition \ref{curve}, we obtain
$$
\frac{\partial^2 \mathcal{J}}{\partial r^2}[0,\xi_{1},\xi_{2}]=-\Tr([\nabla_{\Gamma}\xi_{2}][\nabla_{\Gamma}\xi_{1}])+\Div_{\Gamma}\xi_{1}\cdot \Div_{\Gamma}\xi_{2}+\left([\nabla_{\Gamma}\xi_{1}]\nn\cdot[\nabla_{\Gamma}\xi_{2}]\nn\right).
$$
Notice that $\Tr([\nabla_{\Gamma}\xi_{2}][\nabla_{\Gamma}\xi_{1}])=\Tr([\nabla_{\Gamma}\xi_{1}][\nabla_{\Gamma}\xi_{2}])$.
$$
 \frac{\partial^2 \mathcal{N}}{\partial r^2}[0,\xi_{1},\xi_{2}]=[\nabla_{\Gamma}\xi_{2}][\nabla_{\Gamma}\xi_{1}]\nn+[\nabla_{\Gamma}\xi_{1}][\nabla_{\Gamma}\xi_{2}]\nn-\left([\nabla_{\Gamma}\xi_{1}]\nn\cdot[\nabla_{\Gamma}\xi_{2}]\nn\right)\nn.
$$
 In the last section we give a second method to obtain higher order derivatives using the G\^ateaux derivatives of the surface differential operators.

\begin{remark} The computation of the derivatives does  not require more than the first derivative of the deformations $\xi$. As a consequence for hypersurfaces of class $\mathscr{C}^{k+1}$, it suffices to consider deformations of class $\mathscr{C}^{k+1}$ to conserve the regularity $\mathscr{C}^k$ of the Jacobian and of the normal vector by differentiation.\end{remark}

\subsection{G\^ateaux differentiability of pseudo-homogeneous kernels}
  
   The following theorem  establishes sufficient conditions for the G\^ateaux differentiability  of the boundary integral operators described above.  
  \begin{theorem}\label{PDF} Let $p\in\N$. We set $(\Gamma\times\Gamma)^*=\left\{(x,y)\in\Gamma\times\Gamma; \;x\not=y\right\}$.
  Assume that the following two conditions are satisfied:
  
  1) For all fixed $(x, y)\in(\Gamma\times\Gamma)^*$ the function
$$
\begin{array}{cccl}f:&B^{\infty}(0,\varepsilon)&\rightarrow&\C\\
  &r&\mapsto& k_{r}(y+r(y),x+r(x)-y-r(y))J_{r}(y)
\end{array}
$$
is $\mathscr{C}^{p+1}$-G\^ateaux differentiable.

2) The functions $(y,x-y)\mapsto f(r_{0})(y,x-y)$ and $$(y,x-y)\mapsto d^l f[r_{0},\xi_{1},\hdots,\xi_{l}](y,x-y)$$ are pseudo-homogeneous of class $-m$ for all $r_{0}\in B^{\infty}(0,\varepsilon)$, for all $l=1,\hdots,p+1$ and for all $\xi_{1},\hdots,\xi_{p+1}\in\mathscr{C}^{\infty}(\Gamma,\R^d)$.

Then for any $s\in\R$ the mapping 
$$
\begin{array}{ccl}B^{\infty}(0,\varepsilon)&\rightarrow &\mathscr{L}(H^{s}(\Gamma), H^{s+m}(\Gamma))\\
r&\mapsto&\tau_{r}\mathcal{K}_{\Gamma_{r}}\tau_{r}^{-1}
\end{array}
$$ 
is $\mathscr{C}^{p}$-G\^ateaux differentiable and
\begin{equation*}d^p\left(\tau_{r}\mathcal{K}_{\Gamma_{r}}\tau_{r}^{-1}\right)[r_{0},\xi_{1},\hdots,\xi_{p}]u(x)=\int_{\Gamma}d^pf[r_{0},\xi_{1},\hdots,\xi_{p}](y,x-y)u(y)d s(y).
\end{equation*}
\end{theorem}
\begin{proof}
We  use the  linearity of the integral and Taylor expansion with integral remainder. We do the proof  for $p=1$ only. Let $r_{0}\in B^{\infty}(0,\varepsilon)$, $\xi\in\mathscr{C}^{\infty}(\Gamma,\R^d)$ and $t$ small enough such that $r_{0}+t\xi\in B^{\infty}(0,\varepsilon)$. We have
\begin{equation*}
\begin{split}
 f(r_{0}+t\xi,x,y)-f(r_{0},y,x-y)&=t\frac{\partial f}{\partial r}[r_{0},\xi](y,x-y)\\
 &
 +t^2\int_{0}^1(1-\lambda)\frac{\partial^2 f}{\partial r^2}[r_{0}+\lambda t\xi,\xi](y,x-y)d\lambda.
\end{split}
\end{equation*}
We have to verify that each term in this equality is a kernel of an operator mapping $H^s(\Gamma)$ to  $H^{s+m}(\Gamma)$.  The two first terms in the left hand side are pseudo-homogeneous kernels of class  $-m$ and by hypothesis  $\dfrac{\partial f}{\partial r}[r_{0},\xi]$ is also a kernel of class $-m$. It remains to prove that  the operator with kernel 
$$
 (x,y)\mapsto \int_{0}^1(1-\lambda)\frac{\partial^2 f}{\partial r^2}[r_{0}+\lambda t\xi,\xi](x,y) d\lambda
$$ 
acts from $H^s(\Gamma)$ to $H^{s+m}(\Gamma)$ with norm bounded uniformly in $t$. 
Since $\dfrac{\partial^2 f}{\partial r^2}[r_{0}+\lambda t\xi,\xi]$ is pseudo-homogeneous of class  $-m$ for all $\lambda\in[0,1]$, it suffices to  use Lebesgue's theorem in order to invert the integration with respect to the variable $\lambda$ and the  integration with respect to $y$ on $\Gamma$. 
$$
\begin{array}{ll}
  &\null\hskip-2em\displaystyle{\left\|\int_{\Gamma}\left(\int_{0}^1(1-\lambda)\frac{\partial^2 f}{\partial r^2}[r_{0}+\lambda t\xi,\xi](x,y)d\lambda\right)u(y)d s(y)\right\|_{H^{s+m}(\Gamma)}}\\
  =&\displaystyle{\left\|\int_{0}^1(1-\lambda)\left(\int_{\Gamma}\frac{\partial^2 f}{\partial r^2}[r_{0}+\lambda t\xi,\xi](x,y)u(y)d s(y)\right)d\lambda\right\|_{H^{s+m}(\Gamma)}}\\
  \leq&\sup_{\lambda\in[0,1]}\displaystyle{\left\|\left(\int_{\Gamma}\frac{\partial^2 f}{\partial r^2}[r_{0}+\lambda t\xi,\xi](x,y)u(y)d s(y)\right)\right\|_{H^{s+m}(\Gamma)}}\\
  \leq&C\|u\|_{H^s(\Gamma)}.
\end{array}
$$
We then have 
\begin{equation*}
 \begin{split}
 \dfrac{1}{t}\bigg(\int_{\Gamma}f(r_{0}&+t\xi,x,y)u(y)\,d s(y)-\int_{\Gamma}f(r_{0},x,y)u(y)\,d s(y) \bigg) \\
 &=\int_{\Gamma}\frac{\partial f}{\partial r}[r_{0},\xi](x,y)u(y)d s(y)\\
 &+t\int_{\Gamma}\left(\int_{0}^1(1-\lambda)\frac{\partial^2 f}{\partial r^2}[r_{0}+\lambda t\xi,\xi](x,y)d\lambda\right)u(y)\,d s(y).
 \end{split}
\end{equation*}
We pass to the operator norm limit $t\to0$ and  we obtain the first G\^ateaux derivative. For higher order derivatives it suffices to write the proof with $d^pf[r_{0},\xi_{1},\hdots,\xi_{k}]$ instead of $f$. The linearity, the symmetry and the continuity of the first derivative are deduced from the corresponding properties of the derivatives of the kernel.
\end{proof}

Now we will consider some particular classes of pseudo-homogeneous kernels.
\begin{corollary}\label{corPDF} Assume that the kernels $k_{r}$ are of the form
$$k_{r}(y_{r},x_{r}-y_{r})=G(x_{r}-y_{r})$$
where $G\in\mathscr{C}^{\infty}(\R^d\setminus\{0\})$ is a pseudo-homogeneous kernel of class $-m$, $m\in\N$, which does not depend on $r$. Then the mapping 
$$
\begin{array}{ccl}B^{\infty}(0,\varepsilon)&\rightarrow &\mathscr{L}(H^{t}(\Gamma), H^{t+m}(\Gamma))\\
r&\mapsto&\tau_{r}\mathcal{K}_{\Gamma_{r}}\tau_{r}^{-1}
\end{array}
$$ 
is $\mathscr{C}^{\infty}$-G\^ateaux differentiable and
 the kernel of the first derivative at $r=0$ is defined for $\xi\in\mathscr{C}^{\infty}(\Gamma,\R^d)$ by
\begin{equation*}
  \begin{split} 
  df[0,\xi]=(\xi(x)-\xi(y))\cdot\nabla G(x-y)+G(x-y)\Div_{\Gamma}\xi(y).
  \end{split}
\end{equation*}  
\end{corollary}
\begin{proof}  For  fixed $(x, y)\in(\Gamma\times\Gamma)^*$, consider the mapping 
$$
 f: B^{\infty}(0,\varepsilon)\ni r \mapsto f(r,x,y)=G(x+r(x)-y-r(y))J_{r}(y)\in\C.
$$ 
By Theorem \ref{PDF} we have to prove that $r\mapsto f(r)$ is $\mathscr{C}^{\infty}$-G\^ateaux differentiable and that each derivative defines a pseudo-homogeneous kernel of class $-m$.

 \noindent$\rhd$\underline{Step 1:}\newline
 First we prove that for fixed $(x,y)\in(\Gamma\times\Gamma)^*$ the mapping $r\mapsto f(r,x,y)$ is infinitely G\^ateaux differentiable on $B^{\infty}(0,\varepsilon)$. By  Lemma \ref{J} the mapping $r\mapsto J_{r}(y)$ is infinitely G\^ateaux differentiable on $B^{\infty}(0,\varepsilon)$, the mapping $r\mapsto x+r(x)$ is also infinitely G\^ateaux differentiable on $B^{\infty}(0,\varepsilon)$ and the kernel  $G$ is of class $\mathscr{C}^{\infty}$ on $\R^d\setminus\{0\}$. Being composed of  infinitely G\^ateaux differentiable maps,  the mapping $r\mapsto f(r,x,y)$ is,  too.
  
\noindent $\rhd$\underline{Step 2:}\newline
 We then prove that each derivative defines a  pseudo-homogeneous kernel of class $-m$,  that is to say that for all $p\in\N$ and for any $p$-tuple $(\xi_{1},\hdots,\xi_{p})$ the function
$$ 
 (x,y)\mapsto d^pf[r_{0},\xi_{1},\hdots,\xi_{p}](x,y)
$$ 
is pseudo-homogeneous of class $-m$. 
By formula \eqref{Gsym}, it remains to write the proof for the function $\dfrac{\partial^p}{\partial r^p}f[r_{0},\xi]$ with $\xi\in\mathscr{C}^{\infty}(\Gamma,\R^d)$. The Leibniz  formula gives
$$
  \dfrac{\partial^p}{\partial r^p}f[r_{0},\xi](x,y)
  =
  \sum\limits_{l=0}^p\!\binom{p}{l}\dfrac{\partial^l}{\partial r^l}
  \!\left\{G(x+r(x)-y-r(y))\right\}\![r_{0},\xi] \dfrac{\partial^{p-l} \mathcal{J}}{\partial r^{p-l}}[r_{0},\xi](y).
$$
 Since  $\dfrac{\partial^{p-l} \mathcal{J}}{\partial r^{p-l}}[r_{0},\xi]\in\mathscr{C}^{\infty}(\Gamma,\R)$, we have to prove that  
$$
(x,y)\mapsto \frac{\partial^l}{\partial r^l}\left\{G(x+r(x)-y-r(y))\right\}[r_{0},\xi]
$$ 
defines a pseudo-homogeneous kernel of class  $-m$. We have
\begin{multline*}
\frac{\partial^l}{\partial r^l}\big\{G(x+r(x)-y-r(y))\big\}[r_{0};\xi]\\
=\D^{l}G[x+r_{0}(x)-y-r_{0}(y); \xi(x)-\xi(y),\hdots,\xi(x)-\xi(y)].
\end{multline*}
By definition, $G(z)$ admits the following asymptotic expansion when $z$ tends to zero:
\begin{equation}
\label{delG}
 G(z)=G_{m}(z)+\sum_{j=1}^{N-1}G_{m+j}(z)+G_{m+N}(z)
\end{equation} 
where $G_{m+j}$ is homogeneous of class $-(m+j)$ for $j=0,\hdots,N-1$ and $G_{m+N}$ is of arbitrary regularity. Using  Taylor expansion,  the following result is easy to see:
\begin{lemma} 
Let the kernel $G_{m}(z)$ be homogeneous of class $-m$ and $\xi\in\mathscr{C}^{\infty}(\Gamma,\R^d)$. Then the function 
$$
  (x,y-x)\mapsto D^{l}G_{m}[x+r_{0}(x)-y-r_{0}(y);\xi(x)-\xi(y),\hdots,\xi(x)-\xi(y)]
$$ 
is pseudo-homogeneous of class $-m$.
\end{lemma} 
By taking derivatives in the expansion \eqref{delG} we conclude that \\
$\dfrac{\partial^l}{\partial r^l}\left\{G(x+r(x)-y-r(y))\right\}[r_{0};\xi]$ is  pseudo-homogeneous of class $-m$ too. This ends the proof of the corollary.
\end{proof}

\begin{theorem}
\label{P'DF}
Let $s\in\R$. Let $G(z)$ be a  pseudo-homogeneous kernel of class $-(m+1)$ with $m\in\N$. Let us fix a compact subdomain $K_{p}$ of $\Omega$. Assume that for all  $r\in B^{\infty}(0,\varepsilon_{p})$, we have $k_{r}(y_{r},x-y_{r})=G(x-y_{r})$. Then the mapping
$$
\begin{array}{ccl}B_{\varepsilon}^{\infty}&\rightarrow&\mathscr{L}\left(H^{s-\frac{1}{2}}(\Gamma), H^{s+m}(K_{p})\right)\\r&\mapsto&\mathcal{P}_{r}\tau_{r}^{-1}
\end{array}
$$ 
is infinitely G\^ateaux differentiable and
\begin{multline*}
 d^p (\mathcal{P}_{r}\tau_{r}^{-1})[r_{0},\xi_{1},\hdots,\xi_{p}]u(x)\\=\int_{\Gamma}d^p\left\{G(x-y-r(y))J_{r}(y)\right\}[r_{0},\xi_{1},\hdots,\xi_{p}]u(y)d s(y).
\end{multline*} 
Its first derivative at $r=0$ in the  direction $\xi\in\mathscr{C}^{\infty}(\Gamma,\R^d)$ is the integral operator  denoted by $\mathcal{P}^{(1)}$ with kernel 
\begin{equation*}
 -\xi(y)\cdot\nabla^zG(x-y)+G(x-y)\Div_{\Gamma}\xi(y).
 \end{equation*}
 The operator $\mathcal{P}^{(1)}$ can be extended to a continuous linear operator from $H^{s-\frac{1}{2}}(\Gamma)$ to $H^{s+m}(\Omega)$ and $H_{loc}^{s+m}(\overline{\Omega^c}).$ 
\end{theorem}

\begin{proof} The kernel and its higher order derivatives are of class  $\mathscr{C}^{\infty}$ on $K_{p}$. \newline 
Writing $\Omega$ as an increasing union of compact subsets, we can define a shape derivative on the whole domain $\Omega$.  Let us  look at the first derivative: The term $G(x-y)\Div_{\Gamma}\xi(y)$ has the same regularity as $G(x-y)$ when $x-y$ tends to zero wheareas  $\xi(y)\cdot\nabla G(x-y)$ loses one order of regularity.  As a consequence, since the kernel is of class $-(m+1)$, its first derivative acts from $H^{s-\frac{1}{2}}(\Gamma)$ to $H^{s+m}(\Omega)$ and $H^{s+m}_{loc}(\overline{\Omega^c})$. 
\end{proof}

\begin{remark} We conclude that the boundary integral operators are smooth with respect to the domain whereas the potential operators lose one order of regularity at each derivation. We point out that we do not need more than the first derivative of the deformations $\xi$ to compute the G\^ateaux derivatives of any order of these integral operators.
\end{remark}

\begin{example}(\textbf{Acoustic single layer potential})
\label{Vk}\label{psi} 
 Let $d=2$ or $d=3$ and $s\in\R$. 
 We denote by $\Psi^r_{\kappa}$ the single layer potential  defined for $u_{r}\in H^{s}(\Gamma_{r})$ with the fundamental solution $G_{a}$ of the Helmholtz equation (see Example~\ref{exacoustic})
$$
\Psi_{\kappa}^ru_{r}(x)=\int_{\Gamma_{r}}G_{a}(\kappa,x-y_{r})u_{r}(y_{r})d s(y_{r}),\;x\in\R^d\setminus\Gamma_{r}.
$$ 
Let $V_{\kappa}^r$ its trace on $\Gamma_{r}$
$$
V_{\kappa}^ru_{r}(x)=\int_{\Gamma_{r}}G_{a}(\kappa,x-y_{r})u_{r}(y_{r})d s(y_{r}),\;x\in\Gamma_{r}.
$$
Since $G_{a}$ is pseudo-homogeneous of class $-1$, the mapping
$$
\begin{array}{ccl}
B^{\infty}(0,\varepsilon)&\rightarrow &\mathscr{L}(H^{s}(\Gamma), H^{s+1}(\Gamma))\\
r&\mapsto&\tau_{r}V_{\kappa}^{r}\tau_{r}^{-1}
\end{array}
$$ 
is infinitely  G\^ateaux differentiable. The mapping
$$
\begin{array}{ccl}
B^{\infty}(0,\varepsilon_{p})&\rightarrow &\mathscr{L}\left(H^{s}(\Gamma), H^{s+\frac{1}{2}}(K_{p})\right)\\
r&\mapsto&\tau_{r}\Psi_{\kappa}^r\tau_{r}^{-1}
\end{array}
$$ 
is infinitely differentiable and its first derivative at $r=0$  can be extended to a linear continuous operator from  
$H^{s}(\Gamma)$ to $H^{s+\frac{1}{2}}(\Omega)\cup H^{s+\frac{1}{2}}_{loc}(\overline{\Omega^c})$. 
\end{example}
Similar results can be deduced for the elastic single layer potential.

\begin{example}(\textbf{Acoustic double layer kernel})\label{Dk} Let $d=2$ or $d=3$ and $t\in\R$.
 We denote by $D_{\kappa}^{r}$ the boundary  integral operator defined for $u_{r}\in H^t(\Gamma_{r})$ by
$$D_{\kappa}^{r}u_{r}(x)=\int_{\Gamma_{r}}\nn_{r}(x_{r})\cdot\nabla G_{a}(\kappa,y_{r}-x_{r})u_{r}(y_{r})d s(y_{r}).$$
The mapping
$$\begin{array}{ccl}B^{\infty}(0,\varepsilon)&\rightarrow &\mathscr{L}(H^{t}(\Gamma), H^{t+1}(\Gamma))\\r&\mapsto&\tau_{r}D_{\kappa}^{r}\tau_{r}^{-1}\end{array}$$ is $\mathscr{C}^{\infty}$-G\^ateaux differentiable . 

Indeed    the mapping 
$$
   B^{\infty}(0,\varepsilon)\ni r\mapsto g(r,x,y)=(\tau_{r}\nn_{r})(x)\cdot(x+r(x)-y-r(y))
$$ 
is $\mathscr{C}^{\infty}$ G\^ateaux differentiable and by using a local coordinate system  (see \cite{Potthast1}) we prove (when $d=3$)  that the G\^ateaux derivatives behaves as $|x-y|^2$ when $x-y\to0$. We use the same notations as in the proof of Lemma \ref{J}. Fix $x\in\Gamma$ and set $g_{x}(r,y)=g(r,x,y)$. We have that $g_{x}\in\mathscr{C}^{\infty}(B^{\infty}(0,\varepsilon)\times\Gamma,\R)$. If $\Gamma$ is parametrised by the atlas $(\mathcal{O}_{i},\phi_{i})_{1\leq i\leq p}$ then when $x\in\Gamma_{i}=\phi_{i}(\mathcal{O}_{i})\cap\Gamma$ we can write $x=\phi_{i}(\eta^x_{1},\eta^x_{2})$ where $(\eta^x_{1},\eta^x_{2})\in\mathcal{O}_{i}$. The tangent plane to $\Gamma$ at $x$ is generated by the vectors $e_{1}(x)=\frac{\partial \phi_{i}}{\partial\eta_{1}}(\eta^x_{1},\eta^x_{2})$ and $e_{2}(x)=\frac{\partial \phi_{i}}{\partial\eta_{2}}(\eta^x_{1},\eta^x_{2})$. 
Thus $g_{x}(r,\phi_{i}(\eta_{1},\eta_{2}))$ has the expression
\begin{multline*}
\frac{(\Id+\D r)\frac{\partial \phi_{i}}{\partial\eta_{1}}(\eta^x_{1},\eta^x_{2})\wedge(\Id+Dr)\frac{\partial \phi_{i}}{\partial\eta_{2}}(\eta^x_{1},\eta^x_{2})}{\big|(\Id+Dr)\frac{\partial \phi_{i}}{\partial\eta_{1}}(\eta^x_{1},\eta^x_{2})\wedge(\Id+Dr)\frac{\partial \phi_{i}}{\partial\eta_{2}}(\eta^x_{1},\eta^x_{2})\big|}
\cdot\\
\cdot\big((\Id+r)\circ\phi_{i}(\eta^x_{1},\eta^x_{2})-(\Id+r)\circ\phi_{i}(\eta_{1},\eta_{2})\big)
\end{multline*}
Using Taylor expansion we have when $y\to x$
$$
 g_{x}(r,y)=0+\D g_{x}(r)[x;y-x]+\frac{1}{2}\D^2g_{x}(r)[x;y-x,y-x]+\hdots
$$
Writing $g_{x}(r)=(g_{x}(r)\circ\phi_{i})\circ\phi_{i}^{-1}$, we have for all $r\in B^{\infty}(0,\epsilon)$ that 
$$
 \D g_{x}(r)=\D_{(\eta_{1},\eta_{2})}(g_{x}(r)\circ\phi_{i})\circ\D\phi_{i}^{-1}.
$$ 
By straigthforward computations we obtain that $\D_{(\eta_{1},\eta_{2})}(g_{x}(r)\circ\phi_{i})=0$ for all $r$ \cite{Potthast2}. Thus by differentiation with respect to $r$ we prove that $g_{x}(r,y)$ and all its G\^ateaux derivatives behaves as $|x-y|^2$ when $x-y\to0$.
\end{example}

\section{Shape differentiability  of surface differential operators, \\  application to hypersingular boundary integral operators}

Many classical hypersingular  boundary integral operators can be expressed as compositions of boundary integral operators with pseudo-homogeneous weakly singular kernels and of surface differential operators. Such representations are often used in the numerical implementation of hypersingular boundary integral operators.
Here we use these representations to study the shape derivatives of hypersingular boundary integral operators. To this end, in addition to the shape derivatives of the weakly singular integral boundary integral operators as studied in Section~~\ref{GDiffPH}, we need to determine the G\^ateaux derivatives with respect to deformations of the surface differential operators acting between Sobolev spaces: The tangential gradient is linear and continuous from $H^{t+1}(\Gamma)$ to $\HH^t(\Gamma)$,  the surface divergence is linear and continuous from $\HH^{t+1}(\Gamma)$ to $H^t(\Gamma)$.

\begin{example}
\label{Nk}(\textbf{Acoustic hypersingular kernel}) 
Let $\kappa\in\C$ with $\Im(\kappa)\ge0$  and $d=3$. The hypersingular kernel is the normal derivative of the double layer kernel. We have
\begin{multline*}
\frac{\partial}{\partial\nn(x)}\frac{\partial}{\partial\nn(y)}G_{a}(\kappa,x-y)\\
=-\nn(x)\cdot\nn(y)\Delta G_{a}(\kappa,x-y)+\nn(x)\cdot \Rot^x \big(\nabla^y G_{a}(\kappa,x-y)\wedge\nn(y)\big).
\end{multline*}
When $d=2$, for a scalar function $\varphi$ the term $-\nabla \varphi\wedge\nn$ is the arc-length derivative $\dfrac{d\varphi}{ds}$.
Using integration by parts with respect to the variable $y$ and that for a scalar function $v$ and a vector $\vec{a}\in\R^d$ it holds  $\nn\cdot\rot(v\vec{a})=-(\nabla v\wedge\nn)\cdot\vec{a}$ we obtain for a scalar density $u$
$$
\begin{array}{l}\displaystyle{\int_{\Gamma}\nn(x)\cdot \Rot^x \big(\nabla^y \big(G_{a}(\kappa,x-y)\big)\wedge\nn(y)\big)u(y)ds(y)}\\\hspace{2cm}=-\displaystyle{\int_{\Gamma}\big(\nabla^x \big(G_{a}(\kappa,x-y)\big)\wedge\nn(x)\big)\cdot\big(\nabla^y \big(u(y)\big)\wedge\nn(y)\big)ds(y)}.
\end{array}
$$
Finally we have 
$$
 \begin{aligned}
  \int_{\Gamma}\frac{\partial}{\partial\nn(x)}\frac{\partial}{\partial\nn(y)}&G_{a}(\kappa,x-y)u(y)ds(y)\\
     &=\kappa^{2}\int_{\Gamma}G_{a}(\kappa,x-y)u(y) (\nn(x)\cdot\nn(y))ds(y)\\
     &\quad-\int_{\Gamma}\big(\nabla_{\Gamma}^xG_{a}(\kappa,x-y)\wedge\nn(x)\big)\cdot\big(\nabla_{\Gamma}u(y)\wedge\nn(y)\big) ds(y).
 \end{aligned}
$$
\end{example}
A similar technique can be applied to the elastic hypersingular boundary integral operator using integration by part and G\"unter's tangential derivatives (see \cite{HsiaoWendland, Kupradze}).

\begin{lemma} 
\label{propGunter} 
Let $d=3$ and $\Gamma$ be a closed orientable surface  in $\R^3$. The tangential G\"unter derivative denoted by $\mathcal{M}$ is defined for a vector function $\vv\in\mathscr{C}^1(\Gamma,\C^3)$ by
$$\mathcal{M}\vv=\big(\nabla\vv-(\Div\vv)\cdot\Id_{\R^3}\big)\nn=\frac{\partial}{\partial\nn}\vv-(\Div\vv)\nn+\nn\wedge\Rot\vv.$$
\begin{itemize}
\item[(i)] We set $\nn=(n_{k})_{1\leq k\leq3}$ and $\mathcal{M}_{y}=(m_{jk})_{1\leq j,k\leq3}$. We have
$$m_{jk}=n_{k}(y)\frac{\partial}{\partial y_{j}}-n_{j}(y)\frac{\partial}{\partial y_{k}}=-m_{kj}.$$
\item[(ii)]  For any scalar functions $u,\tilde{u}$ in $\mathscr{C}^1(\Gamma,\C)$ and vector functions $\vv, \tilde{\vv}$ in $\mathscr{C}^1(\Gamma,\C^3)$ there holds the Stokes formula
\begin{equation}\label{Gunter}\int_{\Gamma}(m_{jk}u)\cdot\tilde{u}\,ds=-\int_{\Gamma}u\cdot(m_{jk}\tilde{u})\,ds\text{ and }\int_{\Gamma}(\mathcal{M}\vv)\cdot\tilde{\vv}\,ds=+\int_{\Gamma}\vv\cdot(\mathcal{M}\tilde{\vv})\,ds.\end{equation}
\end{itemize}
\end{lemma}

\begin{example}\label{Tk}(\textbf{Elastic hypersingular kernel}) Let $\omega\in\R$ and $d=3$. Denote by $\rho,\mu$ and $\lambda$ the density and Lam\'e's constants. The hypersingular kernel  is defined by 
$$
 H(x,y)=T_{x}\transposee{\big(T_{y}G_{e}(\kappa, x-y)\big)}
$$
where $G_{e}$ is the fundamental solution of the Navier equation and $T$ is the traction operator defined in Example~\ref{exelastic}. First of all we rewrite the operator $T\uu$ as
\begin{equation}
\label{ecriture1}
 T\uu=2\mu\mathcal{M}\uu+(\lambda+2\mu)(\Div\uu)\nn-\mu\nn\wedge\Rot\uu.\end{equation}
Then we apply the operator $T_{y}$ in the form \eqref{ecriture1} to the tensor $G_{e}(\kappa_{s},\kappa_{p},x-y)$. It follows
\begin{multline*}
\transposee{\big(T_{y}G_{e}(\kappa_{s},\kappa_{p},x-y)\big)}\\
=2\mu\transposee{\big(\mathcal{M}_{y}G_{e}(\kappa, x-y)\big)}-\transposee{\big(\nn(y)\wedge\Rot_{y}G_{a}(\kappa_{s},x-y)\Id_{\R^3}\big)}\\
+\dfrac{(\lambda+2\mu)}{\mu}\,\transposee{\left(\nn(y)\cdot\Div_{y} G_{e}(\kappa_{s},\kappa_{p},x-y)\right)}.
\end{multline*}
\begin{multline*}
\Div_{y} G_{e}(\kappa_{s},\kappa_{p},x-y)\\
 =\transposee{\big(\nabla_{y} G_{a}(\kappa_{s},x-y)\big)}+\dfrac{1}{\kappa_{s}^2}\transposee{\nabla}_{y} \Delta_{y}\big( G_{a}(\kappa_{s},x-y)- G_{a}(\kappa_{p},x-y)\big)\\
 =\dfrac{\kappa_{p}^2}{\kappa_{s}^2}\transposee{\big(\nabla_{y} G_{a}(\kappa_{p},x-y)\big)}
\end{multline*}
$$
\nn(y)\wedge\Rot_{y}G_{a}(\kappa_{s},x-y)\Id_{\R^3}=\left(\mathcal{M}_{y}-\frac{\partial}{\partial\nn(y)}+\nn(y)\cdot\Div_{y}\right)G_{a}(\kappa_{s},x-y)\Id_{\R^3}
$$
In virtue of the property $(i)$ in Lemma~\ref{propGunter} we can write
$$
\begin{array}{ll}\transposee{\Big(\nn(y)\wedge\Rot_{y}G_{a}(\kappa_{s},x-y)\Id_{\R^3}\Big)}=&\left(-\mathcal{M}_{y}-\dfrac{\partial}{\partial\nn(y)}\right)G_{a}(\kappa_{s},x-y)\Id_{\R^3}\vspace{1mm}\\
&+\transposee{\Big(\nn(y)\cdot\transposee{\nabla}_{y}G_{a}(\kappa_{s},x-y)\Big)}
\end{array}
$$
Collecting the equalities we obtain
\begin{multline*}
\transposee{\big(T_{y}G_{e}(\kappa_{s},\kappa_{p},x-y)\big)}\\
=2\mu\transposee{\big(\mathcal{M}_{y}G_{e}(\kappa, x-y)\big)}+\left(\dfrac{\partial}{\partial\nn(y)}+\mathcal{M}_{y}\right)G_{a}(\kappa_{s},x-y)\Id_{\R^3}\\
+\,\nabla_{y} \Big(G_{a}(\kappa_{p},x-y)-G_{a}(\kappa_{s},x-y)\Big)\cdot\transposee{\nn(y)}.
\end{multline*}
 By integration by part and using the properties  $(i)$ and $(ii)$ of Lemma~\ref{propGunter} we obtain that
\begin{multline*}
\int_{\Gamma}\transposee{\big(T_{y}G_{e}(\kappa_{s},\kappa_{p},x-y)\big)}\uu(y)\,ds(y)
=2\mu\int_{\Gamma}G_{e}(\kappa_{s},\kappa_{p}, x-y)\mathcal{M}_{y}\uu(y)\,ds(y)\\
 -\int_{\Gamma}G_{a}(\kappa_{s},x-y)\mathcal{M}_{y}\uu(y)\,ds(y)+\int_{\Gamma}\dfrac{\partial}{\partial\nn(y)}G_{a}(\kappa_{s},x-y)\uu(y)ds(y)\\
 +\int_{\Gamma}\nabla_{y} \Big(G_{a}(\kappa_{p},x-y)-G_{a}(\kappa_{s},x-y)\Big)\big(\nn(y)\cdot\uu(y)\big)\,ds(y).
\end{multline*}
 The kernel of the last term in the right hand side is pseudo-homogeneous of class $-2$.  Thus  $T_{x}$ applied to this term yields a pseudo-homogeneous kernel of class $-1$. Similarly to $\transposee{\big(T_{y}G_{e}(\kappa_{s},\kappa_{p},x-y)\big)}$, the kernel $T_{x}G_{e}(\kappa_{s},\kappa_{p},x-y)$ can be  rewritten in terms of products of weakly singular kernels and the G\"unter derivative $\mathcal{M}_{x}$. Now we apply the operator $T_{x}$ to the  kernels of the second and third terms on the right hand side  in the form
$$
 T_{x}\uu=(\lambda+\mu)\nn\big(\Div_{x}\uu\big)+\mu\left(\frac{\partial}{\partial\nn(x)}+\mathcal{M}_{x}\right)\uu.
$$
We obtain 
\begin{multline*}
T_{x}\Big\{\dfrac{\partial}{\partial\nn(y)}G_{a}(\kappa_{s},x-y)\cdot\Id_{R^3}\Big\}=\mu\dfrac{\partial^2}{\partial\nn(x)\partial\nn(y)}G_{a}(\kappa_{s},x-y)\cdot\Id_{\R^3}\\
 \qquad+\mu\mathcal{M}_{x}\Big(\dfrac{\partial}{\partial\nn(y)}G_{a}(\kappa_{s},x-y)\cdot\Id_{\R^3}\Big)+(\lambda+\mu)\nn(x)\cdot\nabla_{x}^{\top}\dfrac{\partial}{\partial\nn(y)}G_{a}(\kappa_{s},x-y)
\end{multline*}
\begin{multline*}
-T_{x}\Big\{G_{a}(\kappa_{s},x-y)\cdot\Id_{R^3}\Big\}=-\mu\dfrac{\partial}{\partial\nn(x)}G_{a}(\kappa_{s},x-y)\cdot\Id_{\R^3}\\
\qquad -\mu\,\mathcal{M}_{x}\left(G_{a}(\kappa_{s},x-y)\cdot\Id_{\R^3}\right)-(\lambda+\mu)\nn(x)\cdot\transposee{\nabla_{x}}G_{a}(\kappa_{s},x-y)
\end{multline*}
 We use the equality 
$$
\nabla_{x}\dfrac{\partial}{\partial\nn(y)}G_{a}(\kappa_{s},x-y)=\mathcal{M}_{y}\nabla_{x}G_{a}(\kappa_{s},x-y)-\nn(y)\Delta_{y}G_{a}(\kappa_{s},x-y)
$$ and Lemma~\ref{propGunter} to show that (see \cite{HsiaoWendland} pp. 52)
\begin{multline*}
\int_{\Gamma}\nn(x)\cdot\transposee{\nabla_{x}}\dfrac{\partial}{\partial\nn(y)}G_{a}(\kappa_{s},x-y)\uu(y)ds(y)\\
 -\int_{\Gamma}\nn(x)\cdot\transposee{\nabla_{x}}G_{a}(\kappa_{s},x-y)\mathcal{M}_{y}\uu(y)ds(y)\\
 \qquad=\kappa_{s}^2\,\nn(x)\,\int_{\Gamma}G_{a}(\kappa_{s},x-y)(\nn(y)\cdot\uu(y))\,ds(y).
\end{multline*}
 Finally we have
\begin{multline*}
  \int_{\Gamma}T_{x}\transposee{\big(T_{y}G_{e}(\kappa_{s},\kappa_{p},x-y)\big)}\uu(y)\,ds(y)
  \\=2\mu\int_{\Gamma}\big[T_{x}G_{e}(\kappa_{s},\kappa_{p}, x-y)\big]\mathcal{M}_{y}\uu(y)\,ds(y)\\
  +\mu\int_{\Gamma}\dfrac{\partial^2}{\partial\nn(x)\partial\nn(y)}G_{a}(\kappa_{s},x-y)\uu(y)ds(y)\,
  \\-\mu\int_{\Gamma}\frac{\partial}{\partial\nn(x)}G_{a}(\kappa_{s},x-y)\mathcal{M}_{y}\uu(y)\,ds(y)\\
  -\mu\mathcal{M}_{x}\int_{\Gamma}G_{a}(\kappa_{s},x-y)\mathcal{M}_{y}\uu(y)\,ds(y)
  +\mu\mathcal{M}_{x}\int_{\Gamma}\dfrac{\partial}{\partial\nn(y)}G_{a}(\kappa_{s},x-y)\uu(y)ds(y)\\
  +\int_{\Gamma}T_{x}\nabla_{y} \Big(G_{a}(\kappa_{p},x-y)-G_{a}(\kappa_{s},x-y)\Big)\big(\nn(y)\cdot\uu(y)\big)\,ds(y)\\
  +\,\kappa_{s}^2\,(\lambda+\mu)\nn(x)\,\int_{\Gamma}G_{a}(\kappa_{s},x-y)(\nn(y)\cdot\uu(y))\,ds(y).
\end{multline*}

\end{example}
We see that the  boundary integral operator with either the acoustic hypersingular kernel or the  elastic hypersingular kernel are  operators of order $+1$ on the Sobolev spaces $H^t(\Gamma)$ for $t\in\R$. Using the integral representations above, the differentiability properties of these operators can be deduced from the knowledge of the differentiability properties of the surface differential operators. 
Following the same pullback procedure as in Section~\ref{GDiffPH}, the analysis of the hypersingular integral operators is finally reduced to the analysis of the mappings 
$$
\begin{array}{lcl}r&\mapsto&\tau_{r}\nabla_{\Gamma_{r}}\tau_{r}^{-1}\\
 r&\mapsto&\tau_{r}\Div_{\Gamma_{r}}\tau_{r}^{-1}.
\end{array}
$$
Indeed, the G\"unter derivative can be rewritten in terms of these two differential operators:
$$
\mathcal{M}\vv=\big(\nabla_{\Gamma}\vv-(\Div_{\Gamma}\vv)\cdot\Id_{\R^3}\big)\nn.
$$
\medskip
The results are established in the following theorems.

\begin{theorem}\label{nabla}
The mapping $$\begin{array}{cccc}\mathcal{G}:&B^{\infty}(0,\varepsilon)&\rightarrow&\mathscr{L}(H^{t+1}(\Gamma),\HH^{t}(\Gamma))\\&r&\mapsto&\tau_{r}\nabla_{\Gamma_{r}}\tau_{r}^{-1}\end{array}$$ is $\mathscr{C}^{\infty}$-G\^ateaux differentiable and its first derivative  at $r_{0}$ is  defined for 
$\xi\in\mathscr{C}^{\infty}(\Gamma,\R^d)$ by
$$
d\mathcal{G}[r_{0},\xi]u=-[\mathcal{G}(r_{0})\xi]\mathcal{G}(r_{0})u+\big(\mathcal{G}(r_{0})u\cdot[\mathcal{G}(r_{0})\xi]\mathcal{N}(r_{0})\big)\mathcal{N}(r_{0}).
$$
\end{theorem}
\begin{remark}
Note that we can write $d\mathcal{N}[r_{0},\xi]=-[\mathcal{G}(r_{0})\xi]\mathcal{N}(r_{0})$. Since the first derivatives of $\mathcal{N}$ and $\mathcal{G}$ are expressed in terms of $\mathcal{N}$ and $\mathcal{G}$, we can obtain the G\^ateaux derivatives of all orders recursively. 
\end{remark}

\begin{proof}In accordance with the Definition \eqref{G} and Lemma \ref{N}, to prove the $\mathscr{C}^{\infty}$-G\^ateaux differentiability of $\mathcal{G}$ we have to prove the $\mathscr{C}^{\infty}$-G\^ateaux differentiability of the mapping 
$$
 f:B^{\infty}(0,\varepsilon)\ni r\mapsto\left\{u\mapsto \tau_{r}\left(\nabla\widetilde{\tau_{r}^{-1}u}\right)_{|_{\Gamma_{r}}}\right\}\in \mathscr{L}(H^{t+1}(\Gamma),\HH^{t}(\Gamma)).
$$
 For $x\in\Gamma$, we have 
\begin{multline*}
  \tau_{r}\left(\nabla\widetilde{\tau_{r}^{-1}u}\right)_{|_{\Gamma_{r}}}(x)
  =\nabla\left(\widetilde{u}\circ(\Id+r)^{-1}\right)_{|_{\Gamma_{r}}}(x+r(x))\\
  =\transposee{\left(\Id+\D r\right)_{|_{\Gamma_{r}}}^{-1}}(x+r(x))\circ\nabla\widetilde{u}_{|_{\Gamma}}(x),
\end{multline*}
and
$$
 \left(\Id+\D r\right)_{|_{\Gamma_{r}}}^{-1}(x+r(x))=\left[(\Id+\D r)_{|_{\Gamma}}(x)\right]^{-1}.
$$
The mapping $g: B^{\infty}(0,\varepsilon)\ni r\mapsto(\Id+\D r)_{|_{\Gamma}}\in\mathscr{C}^{\infty}(\Gamma)$ is continuous, and $\mathscr{C}^{\infty}$-G\^ateaux differentiable. Its first derivative is  $dg[0,\xi]=[\D\xi]_{|_{\Gamma}}$ and its higher order derivatives vanish. One can easily see that the mapping 
$h:r\in B^{\infty}_{\epsilon}\mapsto \left\{x\mapsto [g(r)]^{-1}(x)\right\}\in\mathscr{C}^{\infty}(\Gamma)$ is also $\mathscr{C}^{\infty}$ G\^ateaux-differentiable and that we have at $r_{0}$ and in the direction $\xi$:
$$
 dh[r_{0},\xi]=-h(r_{0})\circ dg[r_{0},\xi]\circ h(r_{0})=-h(r_{0})\circ[\D\xi]_{|_{\Gamma}}\circ h(r_{0}).
$$
 and 
$$
 d^nh[r_{0},\xi_{1},\hdots,\xi_{n}]=(-1)^n\!\!\!\sum_{ s\in\mathscr{S}_{n}}(\Id+\D r_{0})^{-1}\circ[\tau_{r_{0}}\D\tau_{r_{0}}^{-1}\xi_{ s(1)}]\circ\hdots\circ[\tau_{r_{0}}\D\tau_{r_{0}}^{-1}\xi_{ s(n)}]
$$ 
where $\mathscr{S}_{n}$ is the permutation group of $\{1,\hdots,n\}$.
Finally we obtain the $\mathscr{C}^{\infty}$-G\^ateaux differentiability of $f$ and we have 
$$df[r_{0},\xi]u=-[f(r_{0})\xi]f(r_{0})u.$$
Notice that this result can also be justified by using commutators :  for example at $r=0$ in the direction , we have
$$\frac{\partial}{\partial r}(\tau_{r}\nabla\tau_{r}^{-1}u)[0,\xi]=\frac{\partial}{\partial\xi}(\nabla u)-\nabla\frac{\partial}{\partial\xi}u=-[\nabla\xi]\nabla\uu $$
where $\dfrac{\partial}{\partial\xi}=\xi\cdot\nabla$.

To obtain the expression of the first derivative of $\mathcal{G}$ we have to differentiate the following expression:
$$
 \begin{array}{cl}
 \mathcal{G}(r)u&=(\tau_{r}\nabla_{\Gamma_{r}}\tau_{r}^{-1}u)=\tau_{r}\nabla \left(\widetilde{\tau_{r}^{-1}u}\right)-\left(\tau_{r}\nn_{r}\cdot\left(\tau_{r}\nabla\left( \widetilde{\tau_{r}^{-1}u}\right)\right)\right)\tau_{r}\nn_{r}\\
 &=f(r)u-\left(f(r)u\cdot \mathcal{N}(r)\right)\mathcal{N}(r).
 \end{array}
$$ 
By Lemma \ref{N} and the chain and product rules we have
\begin{align*}
d \mathcal{G}[r_{0},\xi]&=-[f(r_{0})\xi]f(r_{0})u+\left([f(r_{0})\xi]f(r_{0})u\cdot \mathcal{N}(r_{0})\right)\mathcal{N}(r_{0})\\
 &+\left(f(r_{0})u\cdot [\mathcal{G}(r_{0})\xi]\mathcal{N}(r_{0})\right)\mathcal{N}(r_{0})+\left(f(r_{0})u\cdot \mathcal{N}(r_{0})\right)[G(r_{0})\xi]\mathcal{N}(r_{0})
\end{align*}
Combining the first two terms in the right hand side, we get
\begin{align*}
d\mathcal{G}[r_{0},\xi]&=-[\mathcal{G}(r_{0})\xi]f(r_{0})u+\left(f(r_{0})u\cdot\mathcal{N}(r_{0})\right)[\mathcal{G}(r_{0})\xi]\mathcal{N}(r_{0})\\
 &\qquad+\left(f(r_{0})u\cdot [\mathcal{G}(r_{0})\xi]\mathcal{N}(r_{0})\right)\mathcal{N}(r_{0})\vspace{2mm}\\
 &=-[\mathcal{G}(r_{0})\xi]\mathcal{G}(r_{0})u+\left(f(r_{0})u\cdot [\mathcal{G}(r_{0})\xi]\mathcal{N}(r_{0})\right)\mathcal{N}(r_{0}).
\end{align*}
To conclude, it suffices to note that 
$$
\left(f(r_{0})u\cdot [\mathcal{G}(r_{0})\xi]\mathcal{N}(r_{0})\right)=\left(\mathcal{G}(r_{0})u\cdot [\mathcal{G}(r_{0})\xi]\mathcal{N}(r_{0})\right).$$
\end{proof}
 
\begin{theorem}
\label{D}
The mapping $$\begin{array}{cccc}\mathcal{D}:&B^{\infty}(0,\varepsilon)&\rightarrow&\mathscr{L}(\HH^{t+1}(\Gamma),H^t(\Gamma))\\&r&\mapsto&\tau_{r}\Div_{\Gamma_{r}}\tau_{r}^{-1}\end{array}$$ is $\mathscr{C}^{\infty}$-G\^ateaux differentiable  and its first derivative  at $r_{0}$ is  defined for $\xi\in\mathscr{C}^{\infty}(\Gamma,\R^d)$ by
$$
 d\mathcal{D}[r_{0},\xi]\uu=-\Tr([\mathcal{G}(r_{0})\xi][\mathcal{G}(r_{0})\uu])+\left([\mathcal{G}(r_{0})\uu]\mathcal{N}(r_{0})\cdot[\mathcal{G}(r_{0})\xi]\mathcal{N}(r_{0})\right).
$$
\end{theorem}

\begin{proof}For $\uu\in \HH^{t+1}(\Gamma)$ we have $\mathcal{D}(r)\uu=\Tr([\mathcal{G}(r)\uu])$. Then we use the differentiation rules.
\end{proof}

\begin{remark}(i) Since the first derivative of $\mathcal{D}$ is composed of $\mathcal{G}$ and $\mathcal{N}$ and the first  derivative of $\mathcal{J}$ is composed of $\mathcal{J}$ and $\mathcal{D}$, we can obtain an expression of higher order derivatives of the Jacobian recursively. \\
(ii) Denoting by 
\fontdimen16\textfont2=3.5pt
$\mathcal{M}_{\Gamma_{r}}$ \fontdimen16\textfont2=2pt 
the tangential G\"unter derivative on $\Gamma_{r}$,  the formulas $(\#)$ in section \ref{GDiffPH}  can be rewritten as 
$$
\left\{\begin{array}{ccl}\mathcal{W}(r_{0})&=&J_{r_{0}}(\tau_{r_{0}}\nn_{r_{0}}),\vspace{2mm}\\\dfrac{\partial \mathcal{W}}{\partial r}[r_{0},\xi]&=&-J_{r_{0}}\Big(\tau_{r_{0}}\fontdimen16\textfont2=3.5pt\mathcal{M}_{\Gamma_{r_{0}}}\fontdimen16\textfont2=2pt(\tau_{r_{0}}^{-1}\xi)\Big),\vspace{2mm}\\\dfrac{\partial^{m} \mathcal{W}}{\partial r^{m}}[r_{0},\xi]&\equiv&0\text{ for all }m\ge d.
\end{array}\right.
$$

\end{remark}

\begin{remark}\label{Ck}
(\textbf{Electromagnetic hypersingular kernel}) Let $\kappa\in\C$ with $\Im(\kappa)\ge0$ and $d=3$. The electromagnetic hypersingular operator is defined for a tangential density $\jj\in\TT\HH^{t}(\Gamma)$ by
$$
C_{\kappa}\jj(x)=-\frac{1}{\kappa}\int_{\Gamma}\nn(x)\wedge\left(\Rot^x\Rot^x\big(G_{a}(\kappa,x-y)\,\jj(y)\big)\right)ds(y).
$$
Using the identity $\Rot\Rot=-\Delta+\nabla\Div$ we have
\[
\begin{split}
 C_{\kappa}\jj(x)&=-\nn(x)\wedge\displaystyle{\int_{\Gamma}\Big(\kappa\,G_{a}(\kappa,x-y)\cdot\Id_{\R^3}}\\
 &\qquad\quad+\dfrac{1}{\kappa}\nabla^x_{\Gamma}G_{a}(\kappa,x-y)\Div_{\Gamma}\Big)\jj(y)\,ds(y).
\end{split}
\]
This is the operator of the electric field integral equation in electromagnetism. 
The  operator $C_{\kappa}$  is a priori an operator of order $+1$ on the space of tangential vector functions $\TT\HH^{t}(\Gamma)$, but it is well known that this operator is a bounded Fredholm operator on the space of tangential vector fields of mixed regularity  $\TT\HH\sp{-\frac{1}{2}}(\Div_{\Gamma},\Gamma)$, the set of tangential vector fields whose components are in the Sobolev space $H\sp{-\frac{1}{2}}(\Gamma)$ and whose surface divergence is in $H\sp{-\frac{1}{2}}(\Gamma)$. Therefore it is desirable to study the shape differentiability of this operator defined on the shape dependent space $\TT\HH\sp{-\frac{1}{2}}(\Div_{\Gamma},\Gamma)$. For this, the tools presented above are not directly applicable. It is the purpose of the second part \cite{CostabelLeLouer} of our paper to present an alternative strategy  using the Helmholtz decomposition of the space $\TT\HH\sp{-\frac{1}{2}}(\Div_{\Gamma},\Gamma)$.
\end{remark}

 \small
%

\begin{thebibliography}{10}

\bibitem{Charalambopoulos}
{\sc A.~Charalambopoulos}, {\em On the {F}r\'echet differentiability of
  boundary integral operators in the inverse elastic scattering problem},
  Inverse Problems, 11 (1995), pp.~1137--1161.

\bibitem{ColtonKress}
{\sc D.~Colton and R.~Kress}, {\em Inverse acoustic and electromagnetic
  scattering theory}, vol.~93 of Applied Mathematical Sciences,
  Springer-Verlag, Berlin, second~ed., 1998.

\bibitem{CostabelLeLouer}
{\sc M.~Costabel and F.~Le~Lou{\"e}r}, {\em Shape derivatives of boundary
  integral operators in electromagnetic scattering. {P}art {II}: Application to
  scattering by a homogeneous dielectric obstacle},  (2011).

\bibitem{delaBourdonnaye}
{\sc A.~de~La~Bourdonnaye}, {\em D\'ecomposition de {$H\sp {-1/2}\sb {\rm
  div}(\Gamma)$} et nature de l'op\'erateur de {S}teklov-{P}oincar\'e du
  probl\`eme ext\'erieur de l'\'electromagn\'etisme}, C. R. Acad. Sci. Paris
  S\'er. I Math., 316 (1993), pp.~369--372.

\bibitem{DelfourZolesio}
{\sc M.~C. Delfour and J.-P. Zol{\'e}sio}, {\em Shapes and geometries}, vol.~4
  of Advances in Design and Control, Society for Industrial and Applied
  Mathematics (SIAM), Philadelphia, PA, 2001.
\newblock Analysis, differential calculus, and optimization.

\bibitem{DelfourZolesio2}
{\sc M.~C. Delfour and J.-P. Zol{\'e}sio}, {\em Tangential calculus and shape
  derivatives}, vol.~216 of Lecture Notes in Pure and Appl. Math., Dekker, New
  York, 2001.

\bibitem{Eskin}
{\sc G.~I. {\`E}skin}, {\em Boundary value problems for elliptic
  pseudodifferential equations}, Moscow, 1973.

\bibitem{Hadamard}
{\sc J.~Hadamard}, {\em Sur quelques questions du calcul des variations}, Ann.
  Sci. \'Ecole Norm. Sup. (3), 24 (1907), pp.~203--231.

\bibitem{PierreHenrot}
{\sc A.~Henrot and M.~Pierre}, {\em Variation et optimisation de formes},
  vol.~48 of Math\'ematiques \& Applications (Berlin) [Mathematics \&
  Applications], Springer, Berlin, 2005.
\newblock Une analyse g{\'e}om{\'e}trique. [A geometric analysis].

\bibitem{Hormander}
{\sc L.~H{\"o}rmander}, {\em The analysis of linear partial differential
  operators. {III}}, Classics in Mathematics, Springer, Berlin, 2007.
\newblock Pseudo-differential operators, Reprint of the 1994 edition.

\bibitem{HsiaoWendland}
{\sc G.~C. Hsiao and W.~L. Wendland}, {\em Boundary integral equations},
  vol.~164 of Applied Mathematical Sciences, Springer-Verlag, Berlin, 2008.

\bibitem{IvanyshynKress}
{\sc O.~Ivanyshyn}, {\em Shape reconstruction of acoustic obstacles from the
  modulus of the far field pattern}, Inverse Probl. Imaging, 1 (2007),
  pp.~609--622.

\bibitem{IvanyshynKress2}
{\sc O.~Ivanyshyn and R.~Kress}, {\em Nonlinear integral equations in inverse
  obstacle scattering},  (2006), pp.~39--50.

\bibitem{Journe}
{\sc J.-L. Journ{\'e}}, {\em Calder\'on-{Z}ygmund operators, pseudodifferential
  operators and the {C}auchy integral of {C}alder\'on}, vol.~994 of Lecture
  Notes in Mathematics, Springer-Verlag, Berlin, 1983.

\bibitem{KressRundell}
{\sc R.~Kress and W.~Rundell}, {\em Nonlinear integral equations and the
  iterative solution for an inverse boundary value problem}, Inverse Problems,
  21 (2005), pp.~1207--1223.

\bibitem{Kupradze}
{\sc V.~D. Kupradze, T.~G. Gegelia, M.~O. Bashele{\u\i}shvili, and T.~V.
  Burchuladze}, {\em Three-dimensional problems of the mathematical theory of
  elasticity and thermoelasticity}, vol.~25 of North-Holland Series in Applied
  Mathematics and Mechanics, North-Holland Publishing Co., Amsterdam,
  russian~ed., 1979.
\newblock Edited by V. D. Kupradze.

\bibitem{FLL}
{\sc F.~Le~Lou{\"e}r}, {\em Optimisation de formes d'antennes lentilles
  int{\'e}gr{\'e}es aux ondes millim{\'e}triques}, {PhD} in {N}um{\'e}rical
  {A}nalysis, Universit{\'e} de Rennes 1, 2009.
\newblock \newline \texttt{http://tel.archives-ouvertes.fr/tel-00421863/fr/}.

\bibitem{LeugeringetalAMOptim11}
{\sc G.~Leugering, A.~A. Novotny, G.~P. Menzala, and J.~Soko{\l}owski}, {\em On
  shape optimization for an evolution coupled system}, Appl. Math. Optim. 64,
  (2011), pp.~441--466.

\bibitem{Meyer}
{\sc Y.~Meyer}, {\em Ondelettes et op\'erateurs. {II}}, Actualit\'es
  Math\'ematiques. [Current Mathematical Topics], Hermann, Paris, 1990.
\newblock Op{\'e}rateurs de Calder{\'o}n-Zygmund. [Calder{\'o}n-Zygmund
  operators].

\bibitem{Nedelec}
{\sc J.-C. N{\'e}d{\'e}lec}, {\em Acoustic and electromagnetic equations},
  vol.~144 of Applied Mathematical Sciences, Springer-Verlag, New York, 2001.
\newblock Integral representations for harmonic problems.

\bibitem{Potthast2}
{\sc R.~Potthast}, {\em Fr\'echet differentiability of boundary integral
  operators in inverse acoustic scattering}, Inverse Problems, 10 (1994),
  pp.~431--447.

\bibitem{Potthast1}
\leavevmode\vrule height 2pt depth -1.6pt width 23pt, {\em Fr\'echet
  differentiability of the solution to the acoustic {N}eumann scattering
  problem with respect to the domain}, J. Inverse Ill-Posed Probl., 4 (1996),
  pp.~67--84.

\bibitem{Schwartz}
{\sc J.~T. Schwartz}, {\em Nonlinear functional analysis}, Gordon and Breach
  Science Publishers, New York, 1969.
\newblock Notes by H. Fattorini, R. Nirenberg and H. Porta, with an additional
  chapter by Hermann Karcher, Notes on Mathematics and its Applications.

\bibitem{Zolesio}
{\sc J.~Soko{\l}owski and J.-P. Zol{\'e}sio}, {\em Introduction to shape
  optimization}, vol.~16 of Springer Series in Computational Mathematics,
  Springer-Verlag, Berlin, 1992.
\newblock Shape sensitivity analysis.

\bibitem{Taylor2}
{\sc M.~E. Taylor}, {\em Partial differential equations. {I}}, vol.~115 of
  Applied Mathematical Sciences, Springer-Verlag, New York, 1996.
\newblock Basic theory.

\bibitem{Taylor1}
\leavevmode\vrule height 2pt depth -1.6pt width 23pt, {\em Tools for {PDE}},
  vol.~81 of Mathematical Surveys and Monographs, American Mathematical
  Society, Providence, RI, 2000.
\newblock Pseudodifferential operators, paradifferential operators, and layer
  potentials.

\end{thebibliography}

\end{document}